\documentclass[11pt,british]{amsart}

\usepackage[T1]{fontenc}

\usepackage{babel,eucal,url,amssymb,stmaryrd,booktabs,enumerate,amscd,paralist,cite}

\theoremstyle{plain}
\newtheorem{lemma}{Lemma}[section]

\newtheorem{proposition}[lemma]{Proposition}
\newtheorem{corollary}[lemma]{Corollary}
\newtheorem{theorem}[lemma]{Theorem}
\newtheorem{remark}[lemma]{Remark}
\newtheorem{example}[lemma]{Example}

\newtheorem*{ack}{Acknowledgements}

\DeclareMathOperator{\Ric}{Ric}

\newcommand{\Lie}[1]{\operatorname{\textsl{#1}}}
\newcommand{\lie}[1]{\operatorname{\mathfrak{#1}}}

\newcommand{\SO}{\Lie{SO}}

\newcommand{\Un}{\Lie{U}}
\newcommand{\un}{\lie{u}}
\newcommand{\Gtwo}{\ifmmode{{\rm G}_2}\else{${\rm G}_2$}\fi}

 \newcommand{\cyclic}{\mathop{\kern0.9ex{{+}\kern-2.2ex\raise-.28ex\hbox{\Large\hbox
 {$\circlearrowright$}}}}}

\newcommand{\inp}[2]{\langle #1, #2\rangle}

\newcommand{\Wc}[1]{\mathcal W_{#1}}

\newcommand{\Kah}{\mathcal K}

\DeclareMathOperator{\Cur}{\mathcal K}

\def\sideremark#1{\ifvmode\leavevmode\fi\vadjust{\vbox to0pt{\vss
 \hbox to 0pt{\hskip\hsize\hskip1em
 \vbox{\hsize2.5cm\tiny\raggedright\pretolerance10000
 \noindent #1\hfill}\hss}\vbox to8pt{\vfil}\vss}}}%

\newfont{\eusm}{eusm10 scaled \magstep1}
\newfont{\eusmiii}{eusm10 scaled \magstep3}

\newcommand{\comp}{\makebox[7pt]{\raisebox{1.5pt}{\tiny $\circ$}}}

\newcommand{\trace}{\mathop{\rm trace}}

\title{Harmonic G-structures}

\author{J.~C.~Gonz{\'a}lez~D{\'a}vila}
\address[J.~C.~Gonz{\'a}lez~D{\'a}vila]{Department of Fundamental Mathematics\\
  University of La Laguna\\ 38200 La Laguna, Tenerife, Spain}
\email{jcgonza@ull.es}

\author{F.~Mart\'\i n~Cabrera}
\address[F.~Mart\'\i n~Cabrera]{Department of Fundamental Mathematics\\
  University of La Laguna\\ 38200 La Laguna, Tenerife, Spain}
\email{fmartin@ull.es}

\date{\today}

\begin{document}
\maketitle

\begin{abstract}{\indent}
For closed and connected subgroups $G$ of $\SO(n)$, we study the
energy functional on the space of $G$-structures of a (compact)
Riemannian manifold $(M,\langle \cdot, \cdot \rangle),$ where
$G$-structures are considered as sections of the quotient bundle
${\mathcal S\mathcal O}(M)/G.$ Then, we deduce the corresponding
first and second variation formulae and the characterising
conditions for critical points by means of tools closely  related
with the study of $G$-structures. In this direction, we show the
r\^{o}le in the energy functional played by the intrinsic torsion
of the $G$-structure. Moreover, we analyse the particular case
$G=\Lie{U}(n)$ for even-dimensional manifolds. This leads to the
study of harmonic almost Hermitian manifolds and harmonic maps
from $M$ into ${\mathcal S\mathcal O}(M)/U(n).$ \vspace{4mm}

  \vspace{2mm}

\end{abstract}

\tableofcontents

\section{Introduction}{\indent}
The energy of a map between Riemannian manifolds is a functional
which  has been widely studied by diverse authors
\cite{EeLe1,EeLe2,Ur}. Critical points for the energy functional
are called \emph{harmonic maps} and have been characterised by
Eells and Sampson \cite{EeSa} as maps with vanishing \emph{tension
field}.

For a Riemannian manifold $(M,\langle \cdot , \cdot \rangle)$, we
denote by $(T_1 M, \langle \cdot , \cdot \rangle^S)$ its unit
tangent bundle equipped with the Sasaki metric $\langle \cdot ,
\cdot \rangle^S$ (see \cite{Sak}). Looking at  unit vectors fields
as maps $M \to T_1 M$, if $M$ is  compact and oriented,
 one  can consider the energy  functional as defined on the set
$\mathfrak X_1(M)$ of unit vector fields. Critical points for this
functional give rise to the notion of \emph{harmonic unit vector
field}. The condition characterising harmonic unit vector fields
has been obtained by Wiegmink \cite{Wie1} (see also Wood's paper
\cite{Wood}). This has been also extended in a natural way to
sections of sphere  bundles (see \cite{GMS}, \cite{Salvai}) and to
oriented distributions, considered as sections of the
corresponding Grassmann bundle \cite{GGV}.

In \cite{Wood2}, for principal $G$-bundles $Q \to M$ over a
Riemannian manifold $(M, \langle \cdot , \cdot \rangle)$, Wood
considers global sections $\sigma \,: \; M \to Q/H$ of the
quotient bundle $\pi : Q/H \to M$, where $H$ is a closed subgroup
of $G$ such that $G/H$ is reductive. Note that such global
sections are in one-to-one correspondence with the $H$-reductions
of the $G$-bundle $Q \to M$. Likewise, a connection on $Q \to M$
and a $G$-invariant metric on $G/H$ are fixed. Thus, $Q/H$ can be
equipped in a natural way with a metric $\langle \cdot , \cdot
\rangle_{Q/H},$ defined by using the metrics on $M$ and $G/H$. For
such a metric on $Q/H$, the submersion $\pi : Q/H \to M$ is
Riemannian and has totally geodesic fibres. In such conditions,
harmonic sections are characterised as those with vanishing
vertical tension field. This situation arises when the Riemannian
manifold $M$ is equipped with some additional geometric structure,
viewed as reduction of the structure group of the tangent bundle.

In this paper, we consider the particular situation for
$\Lie{G}$-structures defined on an oriented Riemannian
$n$-manifold  $(M ,\langle \cdot , \cdot \rangle)$, where
$\Lie{G}$ is a closed and connected subgroup of $\SO(n).$ The
manifold $M$ is said to be equipped with a $\Lie{G}$-structure if
its oriented orthogonal frame bundle $\mathcal{SO}(M)$ admits a
reduction $\mathcal{G}(M)$ to the subgroup $\Lie{G}$. Moreover, if
$\mathcal{SO}(M)/G = \mathcal{SO}(M) \times_G \Lie{SO}(n)/ G$ is
the quotient bundle under the action of $G$ on $\mathcal{SO}(M),$
the existence of a $G$-structure on $M$ is equivalent to the
existence of a global section $\sigma : M \to \mathcal{SO}(M)/G.$
In the present work, we analyse the energy functional defined on
the space of sections $\Gamma^{\infty} (\mathcal{SO}(M)/G)$ of the
quotient bundle. Thus, if  $\xi^G$ denotes the intrinsic torsion
of the $G$-structure,   we clearly shows the central r\^{o}le played
by $\xi^G$ in the energy functional.
 (Theorem \ref{siginttor}). Furthermore, the first
variation formula is deduced (Theorem \ref{firstvar}). Then, we
show several equivalent characterising conditions  of critical
points for the energy functional on the space of $G$-structures
defined on $(M ,\langle \cdot , \cdot \rangle)$ (Theorem
\ref{carharm}). This gives rise to the notion of \emph{harmonic
$G$-structure} for general Riemannian manifolds, not necessarily
compact and oriented. It is worthwhile to note that harmonic
$G$-structures are not necessarily critical for the energy
functional on all maps from $(M,\langle\cdot,\cdot\rangle)$ to
$(\mathcal{SO}(M)/G,\langle\cdot,\cdot\rangle_{\mathcal{SO}(M)/G})
.$ They are harmonic maps when the corresponding harmonic
$G$-structures satisfy a condition involving the curvature of the
Riemannian manifold. Additionally, we deduce the second variation
formula (Theorem \ref{secondvar}).

We point out that because the intrinsic torsion of the
$G$-structures is involved in all results above mentioned, this
makes possible going further in the study of relations between
harmonicity and classes of $G$-structures. This will be
illustrated in  Section \ref{sect:almherm}, where we focus
attention  on the  study of harmonic almost Hermitian structures
initiated by Wood in \cite{Wood1,Wood2}. Thus, we study
harmonicity of almost Hermitian structures by using the tools
developed in Section \ref{charactgstruc}, recovering Wood's
results and proving some additional ones. In Theorem
\ref{characharmherm1}, several equivalent characterising
conditions  for harmonic almost Hermitian structures are shown.
The relation of harmonicity with Gray-Hervella's classes of almost
Hermitian structures is studied in Theorem \ref{classhermharm}.
Note that the results there contained characterise harmonic almost
almost Hermitian structures by means of conditions on the
Riemannian curvature.
 Concretely, in terms of the particular Ricci tensor $\Ric^*$ determined
 by the almost Hermitian structure.
 Finally, we point out that Theorem \ref{classhermharm}, in some sense,   generalises the
 results
  proved by  Bor et al.
   \cite{BHLS} (see Theorem \ref{bor-Hlam-Salva}). In fact,
   note that the results in \cite{BHLS} are stated
  for conformally flat manifolds, i.e.,  Weyl curvature tensor vanished.

After these remarks, we focus  attention  on the study of
harmonicity as a map of almost Hermitian structures. Results in
that direction were already  obtained by Wood \cite{Wood2}. Here
we complete such results by using tools here presented.

 For completeness, we
finish this paper by briefly giving  a detailed and self-contained
explanation of the situation for nearly K\"{a}hler manifolds.
Thus, we will recover results already known originally proved,
some of them, by Gray and, others, by Wood. However, we will
display alternative proofs in terms of the intrinsic torsion.
Additionally, it is also shown a Kirichenco`s result \cite{Kir}
saying that, for nearly K\"{a}hler manifolds,  the intrinsic
torsion is parallel with respect to the minimal connection.

\begin{ack} {\rm  The authors are are supported by  grants from
MEC (Spain), projects MTM2004-06015-C02-01 and  MTM2004-2644. }
\end{ack}

\section{Preliminaries}{\indent} First we recall some notions relative to
$\Lie{G}$-structures, where $\Lie{G}$ is a subgroup of the linear
group $\Lie{GL}(n , \mathbb R)$. The Lie algebra of $G$ will be
denoted by $\mathfrak{g}$. An $n$-dimensional manifold $M$ is
equipped with a $\Lie{G}$-structure if its frame bundle admits a
reduction $\mathcal{G}(M)$ to the subgroup $\Lie{G}$.  Moreover,
if $(M ,\langle \cdot , \cdot \rangle)$ is an $n$-dimensional
oriented Riemannian manifold, we can consider the principal
$SO(n)$-bundle $\pi_{\Lie{SO}(n)} :  \mathcal{SO}(M) \to M$ of the
oriented orthonormal frames with respect to the metric $\langle
\cdot, \cdot \rangle.$ A $G$-structure on $(M ,\langle \cdot ,
\cdot \rangle)$ is a reduction $\mathcal{G}(M)\subset
\mathcal{SO}(M)$ to a subgroup $\Lie{G}$ of $\SO(n).$

In what follows, we always assume that $G$ is closed and also,
connected. Then, the quotient space $SO(n)/G$ is a homogeneous
manifold and it becomes into a normal homogeneous Riemannian
manifold with bi-invariant metric induced by the inner product
$\langle \cdot,\cdot \rangle$ on ${\mathfrak s \mathfrak o}(n)$
given by $\langle X,Y\rangle  = -\trace XY,$ the natural extension
of the usual Euclidean product $\langle \cdot , \cdot \rangle$ on
$\mathbb{R}^n$ to $\mbox{End}(\mathbb{R}^n).$ Let ${\mathcal
S\mathcal O}(M)/G$ be the orbit space under the action of $G$ on
${\mathcal S\mathcal O}(M)$ on the right as subgroup of $SO(n).$
Then the $G$-orbit map $\pi_{G}:{\mathcal S\mathcal O}(M)\to
{\mathcal S\mathcal O}(M)/G$ is a principal $G$-bundle and we have
$\pi_{SO(n)} = \pi\comp \pi_{G},$ where $\pi: {\mathcal S\mathcal
O}(M)/G\to M$ is a fibre bundle with fibre $SO(n)/G,$ which is
naturally isomorphic to the associated bundle ${\mathcal S\mathcal
O}(M)\times_{SO(n)}SO(n)/G.$ The map $\sigma:M\to {\mathcal
S\mathcal O}(M)/G$ given by $\sigma(m) = \pi_{G}(p),$ for all
$p\in {\mathcal G}(M)$ with $\pi_{SO(n)}(p) = m,$ is well-defined
because $\pi_{G}$ is constant on each fiber of the reduced bundle.
It is a smooth section and we have ${\mathcal G}(M) =
\pi^{-1}_{G}(\sigma(M)).$ Hence, there is a one-to-one
correspondence between the totally of $G$-structures and the
manifold $\Gamma^{\infty}({\mathcal S\mathcal O}(M)/G)$ of all
global sections of ${\mathcal S\mathcal O}(M)/G.$ In what sequel,
we shall also denote by $\sigma$ the $G$-structure determined by
the section $\sigma.$


If $u_{1}= (1,0,\dots,0),\dots ,u_{n} = (0 ,\dots ,0,1)$ is the
canonical orthonormal frame  on $\mathbb{R}^n$, then an oriented
frame $p \in \mathcal{SO}(M)$ can be viewed as an isomorphism $p :
\mathbb{R}^n \to \mbox{\rm T}_{\pi_{\Lie{SO}(n)}(p)}M$ such that
$\{ p(u_1), \dots , p(u_n)\}$ is a positive oriented basis of
$T_{\pi_{\Lie{SO}(n)}(p)}M$. From now on, we will make reiterated
use of the {\it musical isomorphisms} $\flat : \mbox{\rm T}M \to
\mbox{\rm T}^* M$ and $\sharp : \mbox{\rm T}^* M \to \mbox{\rm T}
M$, induced by the metric $\langle \cdot , \cdot \rangle$ on $M$,
respectively defined by $X^{\flat} = \langle X , \cdot \rangle$
and $\langle \theta^{\sharp} , \cdot \rangle = \theta $.

In the presence of  a $G$-structure determined by a section
$\sigma  :  M \to \mathcal{SO}(M)/G$, a frame $p \in
\mathcal{SO}(M)$ is said to be an {\it adapted frame}  to the
$G$-structure, if $p \in \sigma \circ \pi_{\Lie{SO}(n)}(p)$ or,
equivalently,  if $p \in \mathcal{G}(M) \subseteq
\mathcal{SO}(M)$. Note also that,  in a first instance, the bundle
of endomorphisms $\mbox{End} (\mbox{T} M )$ on the fibers in the
tangent bundle $\mbox{T} M$ coincides with the associated vector
bundle $\mathcal{SO}(M) \times_{\Lie{SO}(n)}
\mbox{End}(\mathbb{R}^n)$, where $\Lie{SO}(n)$ acts on
$\mbox{End}(\mathbb{R}^n)$ in the usual way $(g \cdot \varphi)(x)
= g \varphi (g^{-1} x)= (\mbox{Ad}_{\Lie{SO}(n)}(g)\varphi)(x)$.
Thus, it is identified
\begin{equation}
\label{matrix1}
 \varphi_{m} = a_{ji} \, p(u_i)^{\flat} \otimes p(u_j) \cong [(p, a_{ji} \, u_i^{\flat}
\otimes u_j)],
\end{equation}
where $m \in M$ and $p \in \pi^{-1}_{\Lie{SO}(n)}(m)$ and  the
summation convention is used. Such a convention will be followed
in the sequel. When a risk of confusion appear, the sum will be
written in detail.

 In our context, we have also  a reduced subbundle
$\mathcal{G}(M)$. So that we can do the identification $\mbox{End}
(\mbox{T} M ) = \mathcal{G}(M) \times_{G}
\mbox{End}(\mathbb{R}^n)$ because any $\varphi_m$ can be
identified with an element in $\mathcal{G}(M) \times_{G}
\mbox{End}(\mathbb{R}^n)$ as in Equation \eqref{matrix1}, but in
this case must be $p \in \sigma(m)$.

 Now we restrict our attention to the
subbundle $\lie{so}(M)$ of $\mbox{End} (\mbox{T} M )$ of
skew-symmetric endomorphisms $\varphi_m$, for all $m \in M$, i.e.,
$\langle \varphi_m X , Y \rangle= -\langle \varphi_m Y , X
\rangle$. Note that this subbundle $\lie{so}(M)$ is expressed as $
\lie{so}(M) = \mathcal{SO}(M) \times_{\Lie{SO}(n)} \lie{so}(n) =
\mathcal{G}(M) \times_{G} \lie{so}(n) $.  The corresponding
matrices $(a_{ij})$ for $\lie{so}(M)$, given by Equation
\eqref{matrix1}, are such that $a_{ij} =- a_{ji}$. Furthermore,
because $\lie{so}(n)$ is decomposed into the $G$-modules $\lie{g}$
and the orthogonal complement $\lie{m}$ on $\lie{so}(n)$ with
respect to the inner product $\langle\cdot,\cdot\rangle,$ the
bundle $\lie{so}(M)$ is also decomposed into $\lie{so}(M) =
\lie{g}_{\sigma}(M) \oplus \lie{m}_{\sigma}(M) $, where
$\lie{g}_{\sigma}(M) = \mathcal{G}(M) \times_G \lie{g}$ and
$\lie{m}_{\sigma}(M) = \mathcal{G}(M) \times_G \lie{m}$. The
matrices $(a_{ij})$ in Equation \eqref{matrix1} corresponding to
$\lie{g}_{\sigma}(M)$ and $\lie{m}_{\sigma}(M) $ are such that
they are in $\lie{g}$ and $\lie{m}$, respectively. The subindex
$\sigma$ in $\lie{g}_{\sigma}(M)$ and $\lie{m}_{\sigma}(M) $ is to
point out that these bundles are determined by the $G$-structure
$\sigma$. From now on, we will merely write $\lie{g}_{\sigma}$ and
$\lie{m}_{\sigma}$.

Under the conditions above fixed, if $M$ is equipped with a
$G$-structure, then there exists a $G$-connection
$\widetilde{\nabla}$ defined on $M$. Doing the difference
$\widetilde{\xi}_X = \widetilde{\nabla}_X - \nabla_X$, where
$\nabla_X$ is the Levi-Civita connection of $\langle \cdot , \cdot
\rangle$, a tensor $\widetilde{\xi}_X \in \lie{so}(M)$ is
obtained. Because $\nabla$ is torsion-free, $\widetilde{\xi}$ is
an alternative way of giving the torsion of $\widetilde{\nabla}$.
In fact, if $\widetilde{T}$ is the usual torsion tensor of
$\widetilde{\nabla}$ given by $\widetilde{T}(X,Y) =
\widetilde{\nabla}_X Y - \widetilde{\nabla}_Y X -[X,Y]$, then it
is satisfied
\begin{equation} \label{tortor}
\begin{array}{l}
 \widetilde{T}(X,Y) = \widetilde{\xi}_X Y -
\widetilde{\xi}_Y X, \\[1mm]
2 \langle \widetilde{\xi}_X Y , Z \rangle = \langle
\widetilde{T}(X,Y), Z \rangle - \langle \widetilde{T}(Y,Z), X
\rangle + \langle \widetilde{T}(Z,X), Y \rangle.
\end{array}
\end{equation}
 Decomposing $\widetilde{\xi}_X = (
\widetilde{\xi}_X )_{\lie{g}_{\sigma}} + ( \widetilde{\xi}_X
)_{\lie{m}_{\sigma}}$, $( \widetilde{\xi}_X )_{\lie{g}_{\sigma}}
\in \lie{g}_{\sigma}$ and $( \widetilde{\xi}_X
)_{\lie{m}_{\sigma}} \in \lie{m}_{\sigma}$, a new $G$-connection
$\nabla^G$, defined by $\nabla^G_X = \widetilde{\nabla}_X -
(\tilde{\xi}_X )_{\lie{g}_{\sigma}}$, can be considered. Because
the difference between two $G$-connections must be in
$\lie{g}_{\sigma}$, $\nabla^G$ is the unique $G$-connection on $M$
such that its torsion satisfies the condition $\xi^G_X = (
\widetilde{\xi}_X )_{\lie{m}_{\sigma}} = \nabla^{G}_X - \nabla_X
\in \lie{m}_{\sigma}$. $\nabla^G$ is called the {\it minimal
connection} and $\xi^G$ is referred as the {\it intrinsic torsion}
of the $G$-structure $\sigma$ \cite{CleytonSwann:torsion,Salamon}.
A natural way of classifying $G$-structures arises by decomposing
of the space $\mbox{T}^* M \otimes \lie{m}_{\sigma}$ of possible
intrinsic torsions into irreducible $G$-modules. If $\xi^G=0$, the
$G$-structure is usually referred as a {\it parallel} (or {\it
integrable}) $G$-structure. In such a case, the Riemannian
holonomy group of $M$ is contained in $G$.

Associated to the metric connections $\nabla$ and $\nabla^G$ there
are connections one-forms $\omega$ and $\omega^G$  defined on
$\mathcal{SO}(M)$ and $\mathcal{G}(M)$ with values in
$\lie{so}(n)$ and $\lie{g}$, respectively. Note that the
projection $\mathcal{G}(M) \to M$ of the reduced bundle is
$\pi_{\Lie{SO}(n)}$ restricted to $\mathcal{G}(M)$. Therefore, if
$\wp=\{e_1, \dots e_n\} \, : \, U \to \mathcal{G}(M) $ is a local
frame field adapted to the $G$-structure, then
$$
\langle \xi^G_{X} e_i , e_j \rangle_{m} = \langle \nabla^G_{X} e_i
, e_j \rangle_{m} - \langle \nabla_{X} e_i , e_j \rangle_{m} =
\omega^G_{\wp(m)} ( \wp_{\ast m} X )_{ji} - \omega_{\wp(m)} (
\wp_{\ast m} X )_{ji}.
$$
Since the matrices $(\langle \xi^G_{X} e_i , e_j \rangle_{m}) \in
\lie{m}$ and $(\omega^G_{\wp(m)} ( \wp_{\ast m} X )_{ji})\in
\lie{g}$, it is obtained the following identities for matrices
$$
( (\omega_{\wp(m)}( \wp_{\ast m} X)_{ji} )_{\lie{g}} =
(\omega^G_{\wp(m)} ( \wp_{\ast m} X )_{ji}), \qquad
(\omega_{\wp(m)}( \wp_{\ast m} X  )_{ji})_{\lie{m}} = - \left(
\langle \xi^G_{X} e_{i} , e_{j}
 \rangle_{m}\right).
$$
Therefore, the intrinsic torsion is expressed as
\begin{equation} \label{inttorome}
\xi_X^G = - (\omega ( \wp_{\ast } X )_{ji})_{\lie{m}} \,
e_{i}^{\flat} \otimes e_j,
\end{equation}
where $\wp = \{ e_1 , \dots e_n \}$ is a local  frame field
adapted to the $G$-structure.

Finally, we need to point out that,  along the present paper,   we
will consider the  natural extension of the metric $\langle \cdot
, \cdot \rangle$ to $(r,s)$-tensor fields on $M$. Such an
extension is defined by
\begin{equation} \label{extendedmetric}
\langle \Psi,\Phi \rangle = \Psi^{i_{1}\dots i_{r}}_{j_{1}\dots
j_{s}} \Phi^{i_{1}\dots
 i_{r}}_{j_{1}\dots j_{s}},
  \end{equation}
where $\Psi^{i_{1}\dots i_{r}}_{j_{1}\dots j_{s}}$ and
$\Phi^{i_{1}\dots i_{r}}_{j_{1}\dots j_{s}}$ are the components of
$\Psi$ and $\Phi$ with respect to an orthonormal local frame.
\vspace{2mm}

\section{Characterising harmonic $G$-structures via the intrinsic
torsion}{\indent} \label{charactgstruc}
 Now we  consider the bundle $\pi_{G} \, : \,
\mathcal{SO}(M) \to \mathcal{SO}(M)/G$. Because we have $\mbox{T}
\mathcal{SO}(M) = \ker \pi_{\Lie{SO}(n)\ast} \oplus \ker \omega$,
the tangent bundle of $\mathcal{SO}(M)/G$ is decomposed into
$\mbox{T} \mathcal{SO}(M)/G = \mathcal{V} \oplus \mathcal{H}$,
where $\mathcal{V} = \pi_{G\ast} (\ker \pi_{\Lie{SO}(n)\ast} )$
and $\mathcal{H} = \pi_{G\ast} (\ker \omega )$. Then the {\it
vertical} and {\it horizontal} distributions $\mathcal{V}$ and
$\mathcal{H}$ are such that $\pi_{\ast} \mathcal{V} =0$ and
$\pi_{\ast} \mathcal{H} = \mbox{\rm T}M$.

Moreover, we consider the pullback or induced bundle $\pi^*
\lie{so}(M)$ of $\lie{so}(M)$ by $\pi,$ that is, the vector bundle
over $\mathcal{SO}(M)/G$ consisting of those pairs $(pG,
\varphi_m)$, where $\pi(pG)=m$ and $\varphi_m \in \lie{so}(M)_m$.
Alternatively, $\pi^* \lie{so}(M)$ is also described as the
associated bundle $\mathcal{SO}(M) \times_G \lie{so}(n)$ to
$\pi_{G}.$ Then $\pi^* \lie{so}(M)$ is decomposed into $\pi^*
\lie{so}(M) = \lie{g}_{\mathcal{SO}(M)} \oplus
\lie{m}_{\mathcal{SO}(M)}$, where $\lie{g}_{\mathcal{SO}(M)}=
\mathcal{SO}(M) \times_G \lie{g} $ and
 $\lie{m}_{\mathcal{SO}(M)}= \mathcal{SO}(M) \times_G \lie{m}$.
A metric on each fiber of $\pi^* \lie{so}(M)$ is defined by
$$
\langle (pG , \varphi_m) , (pG ,\psi_m) \rangle = \langle
\varphi_m , \psi_m \rangle,
$$
where $\langle \cdot , \cdot \rangle$ in the right side is the
extension to $(1,1)$-tensors of the metric on $M$ given by
\eqref{extendedmetric}. With respect to this metric,  the
decomposition $\pi^* \lie{so}(M) = \lie{g}_{\mathcal{SO}(M)}
\oplus \lie{m}_{\mathcal{SO}(M)}$ is orthogonal.

Additionally, we have a covariant derivative $\nabla$ on $\pi^*
\lie{so}(M)$ induced  by the Levi-Civita connection
  associated to the metric $\langle \cdot , \cdot
\rangle$ on $M$ and given by
\begin{equation} \label{inducedLC}
\left(\nabla_A \tilde{\varphi} \right)_{pG} = \left(pG \, , \,
\frac{\nabla}{ds}_{|s=0} \mbox{pr}_2^{\pi}
\tilde{\varphi}_{\tilde{\gamma}(s)}\right),
\end{equation}
for all $A \in
\mathfrak{X}(\mathcal{SO}(M)/G)=\Gamma^{\infty}(\mbox{T}\mathcal{SO}(M)/G)$
and $\tilde{\varphi} \in \Gamma^{\infty}(\pi^* \lie{so}(M)),$
where $s \to \tilde{\gamma}(s)$ is a curve in $\mathcal{SO}(M)/G$
such that $\tilde{\gamma}(0)=pG$ and $\tilde{\gamma}'(0) = A_{pG}$
and $ \mbox{pr}_2^{\pi}$ is the projection $ \mbox{pr}_2^{\pi}(pG,
\varphi_{m}) = \varphi_{m}$ on $\lie{so}(M).$ Note that, in the
right side, the covariant derivative is along the curve $\gamma(s)
=\pi \circ \tilde{\gamma}(s).$

There is a canonical isomorphism between $\mathcal{V}$ and the
bundle $\lie{m}_{\mathcal{SO}(M) }$. For describing such an
isomorphism, let us firstly say that the elements in
$\lie{m}_{\mathcal{SO}(M)}$ can be seen as pairs $(pG ,
\varphi_m)$ such that if $\varphi_m$ with respect to $p$ is
expressed as in Equation \eqref{matrix1}, then $( a_{ji}) \in
\lie{m}$.
 Now, let us describe the mentioned
canonical isomorphism $\phi_{| \mathcal{V}_{pG}} :
\mathcal{V}_{pG} \to \left( \lie{m}_{\mathcal{SO}(M)}
\right)_{pG}$. For all $a \in \lie{m}$, we have the fundamental
vector field $a^*$ on $\mathcal{SO}(M)$ given by
$$
a^*_p = \frac{d}{dt}_{|t=0} p . \exp t a \in \ker
\pi_{\Lie{SO}(n)*p} \subseteq \mbox{T}_p \mathcal{SO}(M).
$$
Any vector in $\mathcal{V}_{pG}$ is given by $\pi_{G*p} (a^*_p)$,
for some $a =(a_{ji}) \in \lie{m}$. The isomorphism $\phi_{|
\mathcal{V}_{pG}}$ is defined by
$$
\phi_{| \mathcal{V}_{pG}} ( \pi_{G*p} (a^*_p)) = (pG, a_{ji} \,
p(u_i)^{\flat} \otimes p(u_j)).
$$
Next it is extended the map $\phi_{|\mathcal{V}} : \mathcal{V} \to
\lie{m}_{\mathcal{SO}(M)}$ to $\phi : \mbox{T} \,
\mathcal{SO}(M)/G \to \lie{m}_{\mathcal{SO}(M)}$ by saying that
$\phi (A) =0$, for all $A \in \mathcal{H}$,  and $\phi (V) =
\phi_{|\mathcal{V}}(V)$, for all $V \in \mathcal{V}$. This is used
to define a metric $\langle \cdot , \cdot
\rangle_{\mathcal{SO}(M)/G}$ on $\mathcal{SO}(M)/G$ by
\begin{equation} \label{metricquo}
\langle A , B \rangle_{\mathcal{SO}(M)/G} = \langle \pi_{\ast} A ,
\pi_{\ast} B \rangle + \langle \phi (A) , \phi (B) \rangle.
\end{equation}
For this metric, the projection $\pi \, : \, \mathcal{SO}(M)/G \to
M$ is a Riemannian submersion with totally geodesic fibres (see
\cite{Vilms} and \cite[page 249]{Besse:Einstein}). That is, if
${\sf v}  \,: \, \mbox{\rm T} \mathcal{SO}(M)/G \to \mathcal{V}$
and ${\sf h} \,: \, \mbox{\rm T} \mathcal{SO}(M)/G \to
\mathcal{H}$  are respectively the vertical and horizontal
projections and $\nabla^q$ is the Levi-Civita connection of
$\langle \cdot , \cdot \rangle_{\mathcal{SO}(M)/G}$, then $
\nabla^q_{V} W = {\sf v}  \nabla^q_{V} W$ and $ \nabla^q_{V}H
 = {\sf h}  \nabla^q_{V} H$, for all $H \in \Gamma^{\infty}(\mathcal{H})$ and
$V, W \in \Gamma^{\infty}(\mathcal{V})$.

Because $\lie{so}(n) = \lie{g} \oplus \lie{m}$ is a reductive
decomposition, that is, it satisfied ${Ad}_{\Lie{SO}(n)} (G)
\lie{m} \subseteq \lie{m},$ the component $\omega_{\lie{g}}$ in
$\lie{g}$ of the the connection-form $\omega$ is a connection-form
for the bundle $\pi_G \,: \, \mathcal{SO}(M) \to
\mathcal{SO}(M)/G$ which is referred as {\it canonical
connection}. This connection provides a covariant derivative
$\nabla^c$ on $\lie{m}_{\mathcal{SO}(M)}$, which respect to which
the fibre metric is holonomy invariant. The Levi-Civita connection
$\nabla^q$ is related with $\nabla^c$ on
$\lie{m}_{\mathcal{SO}(M)}$ via the projection of the
$\lie{m}$-component of the curvature form $\Omega$ of the
Levi-Civita connection $\nabla$ of $M$. Thus, it is considered the
two-form $\Phi$ on $\mathcal{SO}(M)/G$, with values in
$\lie{m}_{\mathcal{SO}(M)}$, defined by
$$
\Phi (A,B) = \phi \pi_{G*} \Omega ( \tilde{A},
\tilde{B})^*_{\lie{m}} = \phi \pi_{G*}  d \omega ( \tilde{A},
\tilde{B})^*_{\lie{m}} + \phi \pi_{G*}  [\omega (\tilde{A}) ,
\omega ( \tilde{B}) ]^*_{\lie{m}},
$$
where $\tilde{A}, \tilde{B} \in \mbox{T} \mathcal{SO}(M)$ such
that $\pi_{G*} \tilde{A} = A$, $\pi_{G*} \tilde{B} = B$.
Therefore, if on $\mathcal{SO}(M)/G$ we consider the vertical
vectors   $U$ and $V$  and   the horizontal vectors $H$ and $K$,
then
$$
 \Phi (U,V) = 0, \qquad \Phi (U,H) =0, \qquad \Phi (H,K) = \phi \pi_{G*} \Omega (
\tilde{H}, \tilde{K})^*_{\lie{m}}  = \phi \pi_{G*}  d \omega (
\tilde{H},  \tilde{K})^*_{\lie{m}}.
$$

Next, we recall some useful facts proved in \cite[Corollary 2.4
and Proposition 2.7]{Wood2}.
\begin{lemma}[\cite{Wood2}] $\;$ \label{wood:lemma} We have
\begin{enumerate}
\item[{\rm (i)}]
$
\nabla^c_A \tilde{V} = \nabla_A \tilde{V} - [ \phi \, A ,
\tilde{V}].
$
\item[{\rm (ii)}]
$ \phi(\nabla^q_{A}B) - \nabla^c_{A}\phi B = \frac{\textstyle
1}{\textstyle 2}\left \{ [\phi A,\phi B]_{\lie{m}} -
\Phi(A,B)\right \}, $
\end{enumerate}
for all $A,B \in \mathfrak{X}(\mathcal{SO}(M)/G)$ and $\tilde{V}
\in \Gamma^{\infty}(\lie{m}_{\mathcal{SO}(M)}).$
\end{lemma}

From here, we obtain
\begin{equation}\label{nn}
\phi ( \nabla^q_A V) = \nabla^c_A\phi \, V + \frac12 [ \phi A ,
\phi V]_{\lie{m}} = \nabla_A \phi \, V - \frac12 [ \phi A , \phi
V]_{\lie{m}} - [ \phi A , \phi V]_{\lie{g}},
\end{equation}
for all $A \in \mathfrak{X}(\mathcal{SO}(M)/G)$ and $V \in
\Gamma^{\infty}(\mathcal{V}).$
\begin{remark} \label{id:mvert}
{\rm (1) The Lie bracket on $\pi^{*} \lie{so}(M)$ is defined by
$$
[(pG,\varphi_m), (pG,\psi_m)] = (pG , [\varphi_m, \psi_m]) = (pG ,
\varphi_m \circ \psi_m - \psi_m \circ  \varphi_m).
$$

(2) Given a $G$-structure $\sigma \, : \, M \to
\mathcal{SO}(M)/G$, the bundle $\sigma^* \pi^* \lie{so}(M)$ is
identified with $\lie{so}(M)$ by the bijection map $\mbox{\rm
pr}_2^{\pi} \circ \mbox{\rm pr}_2^{\sigma}:
(m,(\sigma(m),\varphi_{m}))\mapsto \varphi_{m}$ and likewise,
$\sigma^* \lie{g}_{\mathcal{SO}(M)} \cong \lie{g}_{\sigma}$ and
$\sigma^* \lie{m}_{\mathcal{SO}(M)} \cong \lie{m}_{\sigma}$. With
respect to sections, if $\varphi \in \Gamma^{\infty}(\lie{so}(M))$
then $pG \to (pG , \varphi_{\pi(pG)})$ belongs to
$\Gamma^{\infty}(\pi^* \lie{so}(M))$ and conversely, if
$\tilde{\varphi} \in \Gamma^{\infty} (\lie{m}_{\mathcal{SO}(M)})$
(respectively, $\tilde{\varphi} \in \Gamma^{\infty}
(\lie{g}_{\mathcal{SO}(M)})$), then $m \to \mbox{pr}_2^\pi
\tilde{\varphi}_{\sigma(m)}$ is in $\Gamma^{\infty}
(\lie{m}_{\sigma})$ (respectively, in $\Gamma^{\infty}
(\lie{g}_{\sigma})$).
 }
\end{remark}

 Now, we consider the set of all possible $G$-structures on a
closed and oriented Riemannian manifold $(M, \langle \cdot , \cdot
\rangle).$ As it has been already mentioned, this set is
identified with the manifold $\Gamma^{\infty}(\mathcal{SO}(M)/G)$
of all global sections $\sigma \, : \, M \to \mathcal{SO}(M)/G$.
Then the {\it energy} of the $G$-structure is defined as the
energy of the corresponding section $\sigma,$ given by the
integral
\begin{equation}\label{1}
 {\mathcal E}(\sigma)=\frac{\textstyle 1}{\textstyle2}\int_{M}\|\sigma_{\ast}\|^{2}dv,
\end{equation}
where $\|\sigma_{\ast}\|^{2}$ is the norm of the differential
$\sigma_{\ast}$ of $\sigma$ with respect to the metrics $\langle
\cdot , \cdot \rangle$ and $ \langle \cdot , \cdot
\rangle_{\mathcal{SO}(M)/G}$, and $dv$ denotes the volume form on
$(M,\langle \cdot , \cdot \rangle)$. On the domain of a local
orthonormal frame field $\{e_1, \dots , e_n \}$ on $M$,
$\|\sigma_{*}\|^{2}$ can be expressed as $\|\sigma_{*}\|^{2} =
\langle
\sigma_{*}e_{i},\sigma_{*}e_{i}\rangle_{\mathcal{SO}(M)/G}$.
Furthermore, from \eqref{1} and using \eqref{metricquo}, it is
obtained that the energy $\mathcal{E}(\sigma)$ of $\sigma$ is
given by
\[
{\mathcal E}(\sigma) = \frac{\textstyle n}{\textstyle 2} {\rm
Vol}(M) + \frac{\textstyle 1}{\textstyle 2}\int_{M}\|  \phi \,
\sigma_*\|^{2}dv.
\]
We will  call the {\it total bending} of the $G$-structure
$\sigma$ to the relevant part of this formula
$B(\sigma)=\frac{\textstyle 1}{\textstyle 2}\int_{M}\| \phi \,
\sigma_*\|^{2}dv$. Because we will show that $ \phi \,
\sigma_{\ast} = - \xi^G$, the total bending provides a measure of
how the $G$-structure $\sigma$ fails to be parallel. Here, we are
doing the identification $\sigma^* \lie{m}_{\mathcal{SO}(M)} \cong
\lie{m}_{\sigma}$ pointed out in Remark \ref{id:mvert}.
\begin{theorem} \label{siginttor}
If $\sigma$ is a global section of $\mathcal{SO}(M)/G$ then $\phi
\, \sigma_{\ast} = - \xi^G$, where $\xi^G$ is the intrinsic
torsion of the $G$-structure determined by $\sigma,$ and the total
bending of the $G$-structure $\sigma$ is given by
$$
B(\sigma)=\frac{\textstyle 1}{\textstyle 2}\int_{M}\| \xi^G\|^{2}
\, dv.
$$
\end{theorem}
\begin{proof} For $X \in \mbox{T}_m M$, we will compute
$\phi \, \sigma_{\ast} X$. If $\wp = \{ e_1 , \dots , e_n \} : U
\to \mathcal{G}(M)$ is a local frame field adapted to the
$G$-structure $\sigma$ with $m\in U,$ then $\pi_{\Lie{SO}(n)}\comp
\wp = Id_{U}$ and taking $\pi_{G|\mathcal{G}(M)} = \sigma \comp
\pi_{\Lie{SO}(n)}$ into account, we have $\sigma_{*} =
\pi_{G*}\comp \wp_{\ast}.$ Therefore, we get
$$
\begin{array}{lcl}
\phi(\sigma_{*} X) & = &  \phi({\sf v}(\sigma_{*}X)) = \phi({\sf
v}(\pi_{G*}\wp_{\ast} X)) = \phi \left(\pi_{G*}(\omega(\wp_{\ast}
X)_{ji} u_i^{\flat} \otimes u_j)^*\right)\\[0.5pc]
 &  = &
\phi\left((\pi_{G*}(\omega(\wp_{\ast} X)_{ji})_{\lie{m}}
u_i^{\flat} \otimes u_j)^*\right) = \left(\sigma(m),
(\omega(\wp_{\ast} X)_{ji})_{\lie{m}}e_i^{\flat} \otimes
e_j\right).
\end{array}
$$
Thus, from \eqref{inttorome}, we have
$$
\phi \sigma_{*} X = (\sigma(m) , -\xi^G_X)
$$
and
$$
\|\phi\sigma_{*}\|^{2} = \langle \phi \, \sigma_* (e_{i })) , \phi
\, \sigma_* (e_{i}) \rangle = \langle (\sigma , \xi^G_{e_i}) ,
(\sigma, \xi^G_{e_i}) \rangle = \langle \xi^G_{e_i} , \xi^G_{e_i }
\rangle =\|\xi^{G}\|^{2}.
$$
Now, the theorem follows using the above identification $\sigma^*
\lie{m}_{\mathcal{SO}(M)} \cong \lie{m}_{\sigma}.$
\end{proof}
Some immediate consequences of last Theorem, most of them already
proved in \cite{Wood2}, are given in the following corollary.
\begin{corollary} The following conditions are equivalent:
\begin{enumerate}
 \item[{\rm (i)}] $\sigma_{\ast} X$ is horizontal, for all $X \in
\mbox{\rm T}M$.
 \item[{\rm (ii)}] $\sigma$ is a parallel $G$-structure, i.e.,  $\xi^G=0$, or $\nabla^G$ is torsion-free.
 \item[{\rm (iii)}] $\sigma$ is an isometric immersion.
  \item[{\rm (iv)}] $\nabla$ can be reduced to a
  $G$-connection.
\end{enumerate}
\end{corollary}

Next, we determine the Euler-Lagrange equation or the critical
point condition for the energy functional $\mathcal{E}$ on closed
and oriented Riemannian manifolds. If we consider a smooth
variation $\sigma_t \in \Gamma^{\infty} (\mathcal{SO}(M)/G)$ of
$\sigma=\sigma_0$, then the corresponding {\it variation field} $m
 \to \varphi(m) = \frac{d}{dt}_{|t=0} \sigma_t(m)$ is a section of the
pullback bundle $\sigma^{*}\mathcal{V}$ over $M$. Thus, the
tangent space $\mbox{T}_{\sigma}  \Gamma^{\infty}
(\mathcal{SO}(M)/G)$ is identified with the space $\Gamma^{\infty}
(\sigma^{*} \mathcal{V})$ of global sections of $\sigma^{*}
\mathcal{V}$ \cite{Ur}. Because $\phi$ determines also an
identification $\sigma^{*}\mathcal{V}\cong \lie{m}_{\sigma}$ by
the bijection
\[
(m,\pi_{G*\sigma(m)}a^{*}_{\sigma(m)})\mapsto \varphi_{m} = a_{ji}
\, p(u_i)^{\flat} \otimes p(u_j),
\]
where $a = (a_{ij})\in \lie{m}$ and $p\in \mathcal{G}(M)$ with
$\pi_{\Lie{SO}(n)}(p) = m,$ we can identify the tangent space
$\mbox{T}_{\sigma}  \Gamma^{\infty} (\mathcal{SO}(M)/G)$ with
$\Gamma^{\infty}(\lie{m}_{\sigma})$.

In following results, we will consider the coderivative $d^*
\xi^G$ of the intrinsic torsion $\xi^G$, which is defined by
$$
(d^* \xi^G)_m (X) = - (\nabla_{e_i} \xi^G)_{e_i} X,
$$
where $\{ e_1, \dots , e_n \}$ is any orthonormal frame on $m\in
M$. In a first instance, $d^* \xi^G$ is a global section of
$\lie{so}(M)= \lie{g}_{\sigma} \oplus \lie{m}_{\sigma}$.
\begin{lemma} \label{coderxilem} The coderivative $d^* \xi^G$ is a global section
of $\lie{m}_{\sigma}$ and is given by
\begin{equation} \label{coderxi}
d^* \xi^G = - (\nabla^{G}_{e_i} \xi^G)_{e_i} - \xi^G_{\xi^G_{e_i}
e_i}.
\end{equation}
\end{lemma}
\begin{proof} Because $\nabla^G = \nabla + \xi^G$, it follows
that $ d^* \xi^G = - (\nabla^{G}_{e_i} \xi^G)_{e_i} + (\xi^G_{e_i}
\xi^G)_{e_i}. $ But one can check that $(\xi^G_{e_i} \xi)_{e_i} =
- \xi^G_{\xi^G_{e_i} e_i}$. Thus, Equation \eqref{coderxi} is
obtained. It is obvious  that $\xi^G_{\xi^G_{e_i} e_i}$ is in
$\lie{m}_{\sigma}$. Since $\nabla^G$ is a $G$-connection,
$\nabla^G$ preserves the $G$-type of a tensor. Therefore, from
$\xi^G_X \in \lie{m}_{\sigma}$, it follows that $(\nabla^{G}_{e_i}
\xi^G)_{e_i} \in \lie{m}_{\sigma}$.
\end{proof}
\begin{theorem}[The first variation formula] \label{firstvar}
If $(M,\langle \cdot , \cdot\rangle)$ a closed and oriented
Riemannian manifold and $\sigma$ a global section of
$\mathcal{SO}(M)/G$, then, for all $\varphi \in \Gamma^{\infty}
(\lie{m}_{\sigma}) \cong \mbox{\rm T}_{\sigma} \Gamma^{\infty}
(\mathcal{SO}(M)/G)$, we have
$$
d \mathcal{E}_{\sigma} (\varphi) = - \int_M \langle \xi^G , \nabla
\varphi \rangle dv =  - \int_M \langle d^* \xi^G , \varphi \rangle
dv,
$$
where $\xi^G$ is the intrinsic torsion of $\sigma$.
\end{theorem}

\begin{proof} We will also denote by $\varphi$ the section in
$\Gamma^{\infty}(\sigma^{*}\mathcal{V})$ which is identified with
$\varphi \in \Gamma^{\infty}(\lie{m}_{\sigma})$, i.e.,
$\mbox{pr}_2^\pi \phi \varphi = \varphi$. If $I_{\varepsilon_1} =
]-\varepsilon_1, \varepsilon_1[ \to \Gamma^{\infty}
(\mathcal{SO}(M)/G)$, $t \to \sigma_t$,  is  a curve such that
$\sigma_0 = \sigma$, and  $ (\sigma_t)'(0) = \varphi$, then we
obtain
\begin{eqnarray*}
d{\mathcal E}_{\sigma}(\varphi) = \frac{\textstyle d}{\textstyle
dt}_{\mid t=0}{\mathcal E}(\sigma_t) & = &
\frac{1}{2}\int_{M}\frac{d}{dt}_{\mid t=0} \langle {\sf v} \,
\sigma_{t*} , {\sf v} \, \sigma_{t*} \rangle_{\mathcal{SO}(M)/G}
dv \\
 & = & \int_{M} \langle  {\sf v} \, \sigma_{*}  ,
\frac{\nabla^q}{dt}_{\mid t=0}  {\sf v} \, \sigma_{t*}
\rangle_{\mathcal{SO}(M)/G} dv.
\end{eqnarray*}
Now, since $\pi$ have totally geodesic fibres and the tangent
vector $(\sigma_t(m))'_{t=0} = \varphi(m)$ of the curve $t \to
\sigma_t(m)$ is vertical, it follows
$$
\langle  {\sf v} \, \sigma_{*}  , \frac{\nabla^q}{dt}_{\mid t=0}
{\sf v} \, \sigma_{t*} \rangle_{\mathcal{SO}(M)/G} = \langle  {\sf
v} \, \sigma_{*}  , \frac{\nabla^q}{dt}_{\mid t=0}  \sigma_{t*}
\rangle_{\mathcal{SO}(M)/G}.
$$
Next, if $I_{\varepsilon_2} = ]-\varepsilon_2, \varepsilon_2[ \to
M$, $s \to \gamma (s)$, is a curve such that $\gamma(0)=m$ and
$\gamma'(0)=X$ and  we consider the smooth map $I_{\varepsilon_1}
\times I_{\varepsilon_2} \to \mathcal{SO}(M)/G$ defined by $(t,s)
\to \sigma_t(\gamma (s))$, then we obtain
\begin{equation*}
\frac{\nabla^q }{\partial t}_{|t=0} \frac{\partial}{\partial
s}_{|s =0} (\sigma_t (\gamma(s)))  = \frac{\nabla^q}{dt}_{|t=0}
(\sigma_{t \ast m} X ) = \frac{\nabla^q}{\partial s}_{|s=0}
\frac{\partial}{\partial t}_{|t =0} (\sigma_t (\gamma(s))) =
\frac{\nabla^q}{d s}_{|s=0} \varphi(\gamma(s)).
\end{equation*}
Therefore,
\begin{eqnarray*}
 \langle  {\sf
v} \, \sigma_{*} X  , \frac{\nabla^q}{dt}_{\mid t=0}  \sigma_{t*}
X \rangle_{\mathcal{SO}(M)/G} & = & \langle  {\sf v} \, \sigma_{*}
X , \frac{\nabla^q}{ds}_{\mid s=0} \varphi(\gamma(s))
\rangle_{\mathcal{SO}(M)/G} \\
& = & \langle  \phi \sigma_{*} X , \phi \frac{\nabla^q}{ds}_{\mid
s=0} \varphi(\gamma(s)) \rangle.
 \end{eqnarray*}
Hence, using \eqref{nn}, we get
\begin{eqnarray*}
 \langle  {\sf
v} \, \sigma_{*} X  , \frac{\nabla^q}{dt}_{\mid t=0}  \sigma_{t*}
X \rangle_{\mathcal{SO}(M)/G} & = &  \langle  \phi \sigma_{*} X ,
\frac{\nabla}{ds}_{\mid s=0} \phi \varphi(\gamma(s)) - \frac12 [
\phi  \sigma_{*} X , \phi \varphi]_{\lie{m}} \rangle
\\
  & = & \langle  \phi \sigma_{*} X ,
\frac{\nabla}{ds}_{\mid s=0}  \phi \varphi(\gamma(s)) \rangle,
 \end{eqnarray*}
where we have used that $SO(n)/G$ is a normal homogeneous
Riemannian manifold
 and $\frac{\nabla}{ds}_{\mid s=0} \phi
\varphi(\gamma(s))$ means the covariant derivative along the curve
$ s \to \sigma (\gamma(s))$. Finally, since by Equation
\eqref{inducedLC} we have
$$
\frac{\nabla}{ds}_{\mid s=0}  \phi \varphi(\gamma(s)) = \left(
\sigma (\gamma(0)) , \frac{\nabla}{ds}_{\mid s=0} \mbox{pr}_2^\pi
\phi \varphi(\gamma(s)) \right),
$$
then we obtain
\begin{eqnarray*}
 \langle  {\sf v} \, \sigma_{*} X  , \frac{\nabla^q}{dt}_{\mid t=0}  \sigma_{t*}
X \rangle_{\mathcal{SO}(M)/G} & = &  \langle  \mbox{pr}_2^\pi \phi
\sigma_{*} X , \frac{\nabla}{ds}_{\mid s=0} \mbox{pr}_2^\pi \phi
\varphi(\gamma(s))  \rangle
\\
  & = &- \langle  \xi^G_{X} , \nabla_X  \mbox{pr}_2^\pi \phi \varphi \rangle.
 \end{eqnarray*}
From this, and taking into account that $\mbox{pr}_2^\pi \phi
\varphi = \varphi$, we will get the required identity
\begin{equation} \label{difener}
d{\mathcal E}_{\sigma}(\varphi) =   - \int_{M} \langle \xi^G ,
\nabla  \varphi \rangle dv.
\end{equation}
On the other hand, we have the equality
$$
\langle \xi^G , \nabla \varphi \rangle = \mbox{div}
(\xi^G)^{\mbox{t}} \varphi + \langle d^* \xi^G ,  \varphi \rangle,
$$
where $\mbox{t}$ means the {\it transpose} operator which is
applied to any section $\Psi \in \Gamma^{\infty} (\mbox{T}^{*} M
\otimes \lie{so}(M))$ and defined  by $\Psi^{\mbox{t}} :
\lie{so}(M) \to \lie{X}(M)$, $\langle \Psi^{\mbox{t}} \varphi, X
\rangle = \langle \Psi_X , \varphi \rangle$. Using last identity
in Equation \eqref{difener}, we will finally  have the another
expression for $d{\mathcal E}_{\sigma}(\varphi)$ required in
Theorem.
\end{proof}
\begin{theorem} \label{carharm}
Under the same assumptions as in Theorem \ref{firstvar}, the
following conditions are equivalent:
\begin{enumerate}
  \item[{\rm (i)}]$\sigma$ is a critical point for the energy functional
   on $\Gamma^{\infty}(\mathcal{SO}(M)/G)$.
    \item[{\rm (ii)}]$d^* \xi^G=0$.
  \item[{\rm (iii)}]$(\nabla^G_{e_i} \xi^G)_{e_i} = -
  \xi^G_{\xi^G_{e_i} e_i}$.
    \item[{\rm (iv)}] If $T^G$ is the torsion of the minimal connection
     $\nabla^G$, then
   \begin{enumerate}
     \item[{\rm (a)}] $\langle (\nabla_{e_i} T^G)(X,Y), e_i\rangle=0$, for all $X,Y \in
    \mathfrak{X}(M)$, and
     \item[{\rm (b)}]$d^* T^G$ is a skew-symmetric endomorphism,
     i.e.,  $d^* T^G \in \lie{so}(M)$.
    \end{enumerate}
\end{enumerate}
\end{theorem}
\begin{proof} An immediate consequence of Theorem \ref{firstvar}
and Lemma \ref{coderxilem} is that conditions (i) and (ii) are
equivalent. The equivalence of (iii) follows from Equation
\eqref{coderxi}. Finally, the equivalence of the conditions in
(iv) is a direct consequence of the identity
$$
2 \langle (\nabla_X \xi^G)_Y Z , U \rangle =  \langle Y, (\nabla_X
T^G)( Z , W) \rangle -  \langle Z, (\nabla_X T^G)( W , Y) \rangle
+  \langle W, (\nabla_X T^G)( Y , Z) \rangle.
$$
\end{proof}
For general  Riemannian manifolds $(M,\langle \cdot ,\cdot
\rangle)$, not necessarily closed and oriented,  we will say that
 a $G$-structure $\sigma$ is {\em harmonic}, if it satisfies $d^* \xi^G=0$.
\begin{theorem}[The second variation formula] \label{secondvar}
With the same assumptions as in Theorem \ref{firstvar}, if
$\sigma$ is a harmonic $G$-structure,  then the Hessian form
$({\rm Hess}\;{\mathcal E})_{\sigma}$ on $
\Gamma^{\infty}(\lie{m}_{\sigma} ) \cong \mbox{T}_{\sigma}
\Gamma^{\infty} (\mathcal{SO}(M)/G)$ is given by
\[
({\rm Hess}\; {\mathcal E})_{\sigma}\varphi =  \int_{M} \left(
\|\nabla  \varphi \|^{2} - \frac12 \| [ \xi^G ,
\varphi]_{\lie{m}_{\sigma}}
  \|^2 +  \langle \nabla
\varphi , 2[ \xi^G , \varphi] -[ \xi^G ,
\varphi]_{\lie{m}_{\sigma}}\rangle\right ) dv.
\]
In particular, if $[\lie{m},\lie{m}]\subset \lie{g}$ or
equivalently $SO(n)/G$ is locally symmetric, then
\[
({\rm Hess}\; {\mathcal E})_{\sigma}\varphi =  \int_{M} \left(
\|\nabla  \varphi \|^{2} - 2 \| [ \xi^G , \varphi]\|^2 \right )
dv.
\]
\end{theorem}
\begin{proof} From results contained in the proof of
Theorem \ref{firstvar}, relative to the first variation formula,
we have
\begin{eqnarray*}
\frac{\textstyle d}{\textstyle dt}_{\mid t=0} d{\mathcal
E}_{\sigma_{t}}(\varphi)  & = & \int_{M} \frac{\textstyle
d}{\textstyle dt}_{\mid t=0} \langle  {\sf v} \, \sigma_{t*}  ,
\frac{\nabla^q}{dt}_{\mid t=t}   \, \sigma_{t*}
\rangle_{\mathcal{SO}(M)/G} dv.
\end{eqnarray*}
But using the same arguments as in the referred proof, we will get
$$
\frac{\textstyle d}{\textstyle dt}_{\mid t=0} \langle  {\sf v} \,
\sigma_{t*} X , \frac{\nabla^q}{dt}_{\mid t=t}    \, \sigma_{t*} X
\rangle_{\mathcal{SO}(M)/G} = \frac{\textstyle d}{\textstyle
dt}_{\mid t=0} \langle  {\sf v} \, \sigma_{t*X}
,\frac{\nabla^q}{ds}_{\mid s=0} \varphi_t(\gamma(s))
\rangle_{\mathcal{SO}(M)/G},
$$
where $s \to \gamma (s)$ is a curve in $M$ such that $\gamma(0)=m$
and $\gamma'(0)=X$, $(\sigma_t)'_{t=t_0} (m) = \varphi_{t_0}(m)$,
and $\frac{\nabla^q}{ds}_{\mid s=0}$ is the covariant derivative
along the curve $s \to \sigma_t(\gamma(s))$. From last identity,
using that the fibers are totally geodesic, it follows
 \begin{eqnarray*}
\frac{\textstyle d}{\textstyle dt}_{\mid t=0} \langle  {\sf v} \,
\sigma_{t*} X , \frac{\nabla^q}{dt}_{\mid t=t}    \, \sigma_{t*} X
\rangle_{\mathcal{SO}(M)/G} & = &   \| {\sf v}
\frac{\nabla^q}{ds}_{\mid s=0} \varphi (\gamma(s)) \|^2_{\mathcal{SO}(M)/G}, \\
&&
 +  \langle  {\sf v} \, \sigma_{*} X
, \frac{\nabla^q}{dt}_{\mid t=0} \frac{\nabla^q}{ds}_{\mid s=0}
\varphi_t(\gamma(s)) \rangle_{\mathcal{SO}(M)/G}.
\end{eqnarray*}
Now, from \eqref{nn}, the first summand is expressed as
\begin{equation} \label{secondvar:1}
 \| {\sf v}
\frac{\nabla^q}{ds}_{\mid s=0} \varphi (\gamma(s))
\|^2_{\mathcal{SO}(M)/G} =  \| \phi \frac{\nabla^q}{ds}_{\mid s=0}
\varphi (\gamma(s)) \|^2 = \|  \nabla_X  \varphi + \frac12[\xi^G_X
,  \varphi]_{\lie{m}_{\sigma}} + [\xi^G_X ,
\varphi]_{\lie{g}_{\sigma}}   \|^2
\end{equation}
and the second summand can be given by
\begin{eqnarray*}
\langle  {\sf v} \, \sigma_{*} X , \frac{\nabla^q}{dt}_{\mid t=0}
\frac{\nabla^q}{ds}_{\mid s=0} \varphi_t(\gamma(s))
\rangle_{\mathcal{SO}(M)/G} & = & \langle  {\sf v} \, \sigma_{*} X
,
 \frac{\nabla^q}{ds}_{\mid s=0} \frac{\nabla^q}{dt}_{\mid t=0}
\varphi_t(\gamma(s)) \rangle_{\mathcal{SO}(M)/G}\\
& & + \langle   R^q{( \varphi (m), {\sf v} \sigma_{*} X)}
\varphi(m), {\sf v} \, \sigma_{*} X
 \rangle_{\mathcal{SO}(M)/G},
\end{eqnarray*}
where $R^q{(A,B)} = \nabla^q_{[A,B]} - [\nabla^q_A , \nabla^q_B]$
 is the Riemannian curvature tensor of $\langle \cdot , \cdot
\rangle_{\mathcal{SO}(M)/G}$ and we have used that $\pi$ has
totally geodesic fibres. On one hand, by using similar arguments
as in the proof of Theorem \ref{firstvar}, we get
\begin{eqnarray} \label{secondvar:2}
\langle  {\sf v} \, \sigma_{*} X ,
 \frac{\nabla^q}{ds}_{\mid s=0} \frac{\nabla^q}{dt}_{\mid t=0}
\varphi_t(\gamma(s)) \rangle_{\mathcal{SO}(M)/G} & = &  \langle
\phi \, \sigma_{*} X , \phi  \frac{\nabla^q}{ds}_{\mid s=0}
(\sigma_t)''_{t=0}(\gamma(s)) \rangle \\
& = & - \langle \xi_X^G , \nabla_X \mbox{pr}^{\pi}_2
(\sigma_t)''_{t=0} (m) \rangle. \nonumber
\end{eqnarray}
Additionally, since $\sigma$ is harmonic, $d^* \xi^G=0$,  we have
the  identity
\begin{equation} \label{secondvar:3}
\langle \xi^G , \nabla \mbox{pr}^{\pi}_2 (\sigma_t)''_{t=0} (m)
\rangle = \mbox{div} (\xi^G)^{\mbox{t}} \mbox{pr}^{\pi}_2
(\sigma_t)''_{t=0} (m).
\end{equation}
On the other hand, in order to compute $\langle   R^q{( \varphi
(m), {\sf v} \sigma_{*} X)} \varphi(m), {\sf v} \, \sigma_{*} X
 \rangle_{\mathcal{SO}(M)/G},$ note that the ${\sf v} \nabla^q_{\varphi} \psi$ is a well defined
connection on the fibres of $\pi$. In our case, ${\sf v}
\nabla^q_{\varphi} \psi=  \nabla^q_{\varphi} \psi$ and  the
corresponding Riemannian curvature  tensor $R^{\sf v}$ is such
that $R^{\sf v}{(\varphi ,\psi_1)} \psi_2 = R^q{(\varphi ,\psi_1)}
\psi_2$. Therefore,
\begin{eqnarray*} \langle   R^q{( \varphi (m), {\sf v} \sigma_{*} X)}
\varphi(m), {\sf v} \, \sigma_{*} X
 \rangle_{\mathcal{SO}(M)/G} &  = & \langle  R^{\sf v}{( \varphi (m), {\sf v} \sigma_{*} X)} \varphi(m) ,
 {\sf v} \, \sigma_{*} X
 \rangle_{\mathcal{SO}(M)/G} .
\end{eqnarray*}
Now, using \eqref{nn}, we get
\begin{eqnarray*}
 \phi {\sf v} \nabla^q_{\varphi(m)} {\sf v} \nabla^q_{{\sf v} \sigma_{*} X} \varphi & = &
 \frac14 [\phi \varphi , [\phi \sigma_{*} X ,
 \phi \varphi]_{\lie{m}} ]_{\lie{m}}
 + \frac12 [\phi \varphi , [\phi \sigma_{*} X ,
 \phi \varphi]_{\lie{g}} ]_{\lie{m}} \\
 &&
  + \frac12  [\phi \varphi ,[\phi \sigma_{*} X,
 \phi \varphi]_{\lie{m}}]_{\lie{g}}
    +   [\phi \varphi ,[\phi \sigma_{*} X,
 \phi \varphi]_{\lie{g}}]_{\lie{g}}, \\
\phi {\sf v} \nabla^q_{{\sf v} \sigma_{*} X} {\sf v}
\nabla^q_{\varphi(m)} \varphi & = & 0, \\
 \phi  {\sf v} \nabla^q_{[\varphi, \sigma_{*} X]} \varphi & = &
-  \frac12 [\phi \varphi , [\phi \sigma_{*} X ,
 \phi \varphi]_{\lie{m}} ]_{\lie{m}}
 - \frac12 [\phi \varphi , [\phi  \sigma_{*} X ,
 \phi \varphi]_{\lie{g}} ]_{\lie{m}},\\
& &-   [\phi \varphi , [\phi \sigma_{*} X ,
 \phi \varphi]_{\lie{m}} ]_{\lie{g}}
 -  [\phi \varphi , [\phi  \sigma_{*} X ,
 \phi \varphi]_{\lie{g}} ]_{\lie{g}}.
\end{eqnarray*}
From these identities and because with the induced metric by the
inner product $\langle \cdot,\cdot \rangle$ on ${\mathfrak s
\mathfrak o}(n)$ is bi-invariant, it is not hard to deduce
\begin{eqnarray} \label{secondvar:4}
 \qquad \quad\langle  R^{\sf v}{( \varphi (m), {\sf v}
\sigma_{*} X)} \varphi(m) ,
 {\sf v} \, \sigma_{*} X
 \rangle_{\mathcal{SO}(M)/G}   & = & - \frac34 \| [\xi_X^G ,
 \varphi]_{\lie{m}_{\sigma}} \|^2
 - \langle  [  \varphi , [ \xi_X^G ,
  \varphi]_{\lie{g}_{\sigma}} ]_{\lie{m}_{\sigma}} , \xi^G_X
  \rangle\\
\nonumber  & = &- \frac34 \| [\xi_X^G ,
 \varphi]_{\lie{m}_{\sigma}} \|^2
 - \|[\xi_X^G ,
 \varphi]_{\lie{g}_{\sigma}} \|^2.
\end{eqnarray}
The required formula for the second variation follows from
\eqref{secondvar:1}, \eqref{secondvar:2}, \eqref{secondvar:3} and
\eqref{secondvar:4}. For the last part of the theorem, we use that
$\nabla \varphi = \nabla^{G}\varphi - [\xi,\varphi]$ and
$\nabla^{G}\varphi\in \Gamma^{\infty}(\lie{m}_{\sigma}).$
\end{proof}

For studying harmonicity as a map of $G$-structures, we need to
consider   $\nabla \sigma_* $, where $\left( \nabla_X \sigma_*
\right)(Y) = \nabla^q_{X} \sigma_* Y - \sigma_* (\nabla_X Y),$ for
all $X,Y\in \mathfrak{X}(M).$ Here as before, $\nabla^q$ also
denotes the induced connection on $\sigma^*T\mathcal{SO}(M)/G.$
\begin{lemma} \label{curvphi1} If $R{(X,Y)} = \nabla_{[X,Y]} - [ \nabla_X , \nabla_Y]$ is the
curvature Riemannian tensor of $(M , \langle \cdot , \cdot
\rangle)$ and $\sigma$ is a $G$-structure on $M$, then
$$
\sigma^* \Phi (X,Y) =  - (\nabla_{X} \xi^G)_Y + (\nabla_{Y}
\xi^G)_X - 2[\xi^G_X , \xi^G_Y] + [ \xi_X^G , \xi^G_Y
]_{\lie{m}_{\sigma}} = - R(X,Y)_{\lie{m}_{\sigma}}.
$$
\end{lemma}
\begin{proof} If $\wp : U \to \mathcal G(M)$ is a local section of
the reduced bundle $\mathcal G(M) \subseteq \mathcal{SO}(M)$, then
$$
\Phi_{\sigma(m)} (\sigma_* X, \sigma_* Y) = \phi \pi_{G*\wp(m)}
\Omega ( \wp_*{X}, \wp_*Y)^*_{\lie{m}} = (\sigma(m) , -
R(X,Y)_{\lie{m}_{\sigma}} ) = - R(X,Y)_{\lie{m}_{\sigma}}
$$
(see \cite[Proposition 4.5]{LM}). Now, if we use $\nabla^G =
\nabla + \xi^G$ in the expression for $R$, it is not hard to see
that
$$
R{(X,Y)} = R^G{(X,Y)} + (\nabla^G_{X} \xi^G)_Y - (\nabla^G_{Y}
\xi^G)_X + \xi^G_{ \xi^G_X Y}  - \xi^G_{\xi^G_YX} - [ \xi^G_X ,
\xi^G_Y ],
$$
where $ R^G{(X,Y)} = \nabla^G_{[X,Y]} - [\nabla_X^G , \nabla_Y^G
]$. Finally, since $R^G \in \Lambda^2 T^* M \otimes
\lie{g}_{\sigma}$, $\xi^G \in T^* M \otimes \lie{m}_{\sigma}$ and
$\nabla^G$ is a $G$-connection, we get
\begin{eqnarray*}
R(X,Y)_{\lie{m}_{\sigma}} & = & (\nabla^G_{X} \xi^G)_Y -
(\nabla^G_{Y} \xi^G)_X + \xi^G_{ \xi^G_X Y}  - \xi^G_{\xi^G_YX} -
[ \xi^G_X , \xi^G_Y ]_{\lie{m}_{\sigma}} \\
 & = & (\nabla_{X} \xi^G)_Y -
(\nabla_{Y} \xi^G)_X  + 2 [\xi^G_X , \xi^G_Y] - [ \xi^G_X ,
\xi^G_Y ]_{\lie{m}_{\sigma}}.
\end{eqnarray*}
From all of this, Lemma follows. Finally, note also that
$R(X,Y)_{\lie{g}_{\sigma}} = R^G{(X,Y)} -  [ \xi^G_X , \xi^G_Y
]_{\lie{g}_{\sigma}}$.
\end{proof}
If $\sigma^{*}\Phi=0$, the $G$-structure $\sigma$ is referred as
{\it flat $G$-structure}. By the final remark in the proof of last
Lemma, this notion is characterised by $R{(X,Y)} = R^G{(X,Y)}  - [
\xi^G_X , \xi^G_Y ]_{\lie{g}_{\sigma}} \in S^2 \lie{g}_{\sigma}$.
Therefore, the intrinsic torsion of a flat $G$-structure  has not
contributions in the $G$-components of $R$ orthogonal to $S^2
\lie{g}_{\sigma}$. Thus, $R$ is in the space of algebraic
curvature tensors for manifolds with parallel $G$-structure.


Now we have the tools to show some  results (Theorem
\ref{verthor}, Theorem \ref{harmmap1} and Theorem \ref{superflat})
 which are versions of Wood's
results given  in \cite{Wood2}, expressed in terms of the
intrinsic torsion $\xi^G$ and the Riemannian curvature tensor $R$.
But we firstly recall that a $G$-structure $\sigma$ is said to be
{\it totally geodesic},  if $\nabla \sigma_* =0$. In such a
situation, $\sigma(M)$ is a totally geodesic submanifold of
$\mathcal{SO}(M)/G$. Weaker conditions can be considered by saying
that a $G$-structure $\sigma$ is {\it vertically geodesic} (resp.,
{\it horizontally geodesic}), if the vertical component (resp.,
{\it horizontal component}) of $\nabla \sigma_*$ vanishes.  In
these situations, $\sigma$ send geodesics to path with horizontal
(resp., vertical) acceleration.
\begin{theorem} \label{verthor} If $\sigma$ is a $G$-structure on $(M,\langle
\cdot , \cdot \rangle)$, then:
\begin{enumerate}
 \item[{\rm (a)}] $\phi (\nabla_X \sigma_*)Y = - \frac12 \left( (\nabla_{X} \xi^G)_Y +
 (\nabla_{Y} \xi^G)_X \right)$. Therefore, $\sigma$ is vertically geodesic if and only if
 $(\nabla_{X} \xi^G)_X = 0.$ In particular,  if $\sigma$ is
 vertically geodesic, then $\sigma$ is a harmonic $G$-structure.
  \item[{\rm (b)}] $ 2 \langle \pi_* (\nabla_X \sigma_*)Y , Z
  \rangle = \langle \xi^G_X , R{(Y,Z)} \rangle + \langle \xi^G_Y , R{(X,Z)}
  \rangle$. Therefore, $\sigma$ is horizontally  geodesic if and only
  if $\langle \xi^G_X , R{(Y,Z)} \rangle$ is a skew-symmetric
  three-form. In particular, if $\sigma$ is a flat $G$-structure,
  then $\sigma$ is horizontally  geodesic.
\end{enumerate}
\end{theorem}
\begin{proof}
For (a). Using Lemma \ref{wood:lemma} , we have
\begin{eqnarray*}
\phi \left(\nabla_{X} \sigma_*\right) Y & = &  \nabla^c_{\sigma_*
X} \phi \sigma_*  Y  + \frac12\left\{[\phi\sigma_{*}X,
\phi\sigma_{*}Y]_{\lie{m}} - \Phi( \sigma_* X , \sigma_*
Y)\right\} - \phi \sigma_* (\nabla_X Y)
 \\
 & = & \nabla_{\sigma_*X} \phi \, \sigma_* Y -  [ \phi \sigma_*X
, \phi \sigma_* Y ]+ \frac12 [ \phi \sigma_*X , \phi \sigma_* Y
]_{\lie{m}} \\
&& - \phi \sigma_* (\nabla_X Y) - \frac12 \Phi( \sigma_* X ,
\sigma_* Y).
\end{eqnarray*}
Now, taking  $\phi \sigma_* = - \xi^G$ into account and using
Lemma \ref{curvphi1}, the required identity in (a) follows.

For (b). In \cite[Theorem 3.4 (ii)]{Wood2}, it is proved that
\begin{eqnarray*}
2 \langle \pi_*  \left(\nabla_{X} \sigma^*\right) Y ,  Z \rangle &
= & \langle \phi \sigma_* X , \Phi( \sigma_* Y , \sigma_* Z)
\rangle +  \langle \phi \sigma_* Y , \Phi( \sigma_* X , \sigma_*
Z) \rangle.
\end{eqnarray*}
Since $\phi \sigma_* = - \xi^G$ and $\Phi( \sigma_* Y , \sigma_* Z
\rangle = -R(Y,Z)_{\lie{m}_{\sigma}}$, (b) follows.
\end{proof}

Next, we compute the respective vertical and horizontal components
of the {\it tension field} $\tau(\sigma)= \left( \nabla^q_{e_i}
\sigma_*\right)(e_i)$ used in variational problems \cite{Ur}.
Given a $G$-structure $\sigma$ on a closed Riemannian manifold
$(M,\langle\cdot ,\cdot \rangle),$ the map $(M, \langle \cdot,
\cdot \rangle) \mapsto (\mathcal{SO}(M)/G,
\langle\cdot,\cdot\rangle_{\mathcal{SO}(M)/G})$ is harmonic, i.e.,
$\sigma$ is a critical point for the energy functional on
$\mathcal{C}^{\infty}(M,\mathcal{SO}(M)/G),$ if and only if
$\tau(\sigma)$ vanishes. Because variations vector fields of
smooth variations of $\sigma$ through sections belong to
$\Gamma^{\infty}(\sigma^{*}\mathcal{V}),$ it follows that harmonic
sections are characterised by the vanishing of the vertical
component of $\tau(\sigma)$. By Theorem \ref{verthor}(a), it
follows  $ \phi \tau(\sigma) = - (\nabla_{e_i} \xi^G)_{e_i} = d^*
\xi^G$ which coincides with the above exposed relative to harmonic
$G$-structures. Since, by Theorem \ref{verthor}(b), the horizontal
component of $\tau(\sigma)$ is determined by the horizontal lift
of the vector field metrically equivalent to the one-form $\langle
\xi^G_{e_i} , R{(e_i,X)}\rangle$, then next result follows.
\begin{theorem} \label{harmmap1} A $G$-structure $\sigma$ on a closed and oriented Riemannian manifold
$(M,\langle \cdot , \cdot \rangle)$ is a harmonic map if and only
if $\sigma$ is a harmonic $G$-structure such that $\langle
\xi^G_{e_i} , R(e_i,X)\rangle=0$. Therefore, if $\sigma$ is flat,
then $\sigma$ is a harmonic map if and only if $\sigma$ is a
harmonic $G$-structure.
\end{theorem}

Such a $G$-structure $\sigma$ is said to determine a
\emph{harmonic map}, even when $M$ is possibly non-compact or
non-orientable and if ${\sf v} \left( \nabla_\cdot {\sf v}
\sigma_*\right) \cdot =0$, the $G$-structure $\sigma$ is called
{\it super-flat}.
\begin{theorem} \label{superflat} We have
\[
\phi (\nabla_X {\sf v} \sigma_*)(Y) = - \frac12 \left( (\nabla_{X}
\xi^G)_Y + (\nabla_{Y} \xi^G)_X + R(X,Y)_{\lie{m}_{\sigma}}
\right).
\]
Therefore, $\sigma$ is super-flat if and only if $\sigma$ is flat
and totally geodesic. In particular, a parallel $G$-structure is
super-flat.
\end{theorem}
\begin{proof}
Using Lemma \ref{wood:lemma}, we have
\[
\phi \left(\nabla_{X} {\sf v} \sigma_*\right) Y  =
\phi(\nabla_{X}\sigma_*)  Y - \phi \nabla^{q}_{X}{\sf
h}\sigma_{*}Y = \phi(\nabla_{X}\sigma_*)  Y   + \frac12 \Phi(
\sigma_* X , \sigma_* Y).
\]
Then, the identity follows using Lemma \ref{curvphi1} and Theorem
\ref{verthor}. Finally, note that if $\sigma$ is super-flat, then
the vanishing of the symmetric part for $X$ and $Y$ of $\phi
(\nabla_X {\sf v} \sigma_*)Y$ implies that $\sigma$ is vertically
geodesic. Meanwhile, the vanishing of the skew-symmetric part for
$X$ and $Y$ implies that $\sigma$ is flat.
\end{proof}

Relevant types of diverse $G$-structures are characterised by
saying that its intrinsic torsion $\xi^G$ is metrically equivalent
to a skew-symmetric three-form, that is, $\xi^G_X Y = - \xi^G_Y
X$. Now we will show some facts satisfied by  such $G$-structures.
\begin{proposition} \label{pro:skew}  For   a $G$-structure $\sigma$  such
that $\xi^G_X Y = - \xi^G_{Y} X$, we have:
\begin{enumerate}
 \item[{\rm (i)}] 
  If $[\xi^G_X  , \xi^G_Y] \in \lie{g}_{\sigma}$, for all $X,Y \in
  \mathfrak X(M)$,
  then
    $\langle R_{(X,Y)\lie{m}_{\sigma}} X , Y \rangle  =  2\langle
\xi_{X} Y , \xi_X Y \rangle$. Therefore, $\sigma$ is parallel  if
and only if  $\sigma$ is flat if and only if $\sigma$ is
super-flat.
 \item[{\rm (ii)}] If $\sigma$ is a harmonic $G$-structure, then
$\sigma$ is also a harmonic map.
\end{enumerate}
\end{proposition}
\begin{proof}
For (i). Because the condition $\xi^{G}_X Y= - \xi^{G}_Y X$
implies that $(\nabla_X \xi^{G})_Y Z = - (\nabla_X \xi^{G})_Z Y$
and $(\nabla_X^{G} \xi)_Y Z = - (\nabla^G_X \xi)_Z Y,$ we will get
the required identity in (i) by using the expression for
$R(X,Y)_{\lie{m}_{\sigma}}$ contained in Lemma \ref{curvphi1}.

For (ii). Applying the first Bianchi's identity, we have
$$
 \langle \xi^G_{e_i} ,
R{(e_i,X)} \rangle = \frac13  \langle  \xi^G_{e_i} e_j , e_k
\rangle \left( \langle R{(e_j,e_k)} e_i , X \rangle  + \langle
R{(e_k,e_i)} e_j , X \rangle + \langle R{(e_i,e_j)} e_k , X
\rangle \right)=0.
$$
\end{proof}

In next Section, we will  study  harmonicity of almost Hermitian
metric structures. Such structures are examples of $G$-structures
defined by means of one or several $(r,s)$-tensor fields $\Psi$
which are stabilised under the action of $G$, i.e.,  $g \cdot
\Psi= \Psi$, for all $g \in G$. Moreover, it will be possible
characterise the harmonicity of such $G$-structures by conditions
  given in terms of those tensors $\Psi$.
The \emph{connection Laplacian} (or {\it rough Laplacian})
\cite{LM} $\nabla^* \nabla \Psi$ will play a relevant r\^{o}le  in
such conditions. We recall that
$$
\nabla^* \nabla \Psi = -  \left( \nabla^2 \Psi \right)_{e_i,e_i},
$$
where $\{ e_1, \dots , e_n \}$ is an orthonormal  frame field and
$(\nabla^2\Psi)_{X,Y} = \nabla_X (\nabla_Y \Psi) -
\nabla_{\nabla_XY}\Psi$. Next Lemma provides an expression for
$\nabla^* \nabla \Psi$ in terms of $\nabla^G$ and $\xi^G$ which
will be useful in the sequel.
\begin{lemma} \label{lapstaten}
Let $(M,\langle \cdot , \cdot \rangle)$ be an oriented Riemannian
$n$-manifold  equipped with a $G$-structure, where the Lie group
$G$ is closed, connected and $G \subseteq \Lie{SO}(n)$. If
 $\Psi$ is a $(r,s)$-tensor field on $M$ which is stabilised under the action of $G$, then
$$
\nabla^* \nabla \Psi =  \left( \nabla^{\Lie{G}}_{e_i}
\xi^{\Lie{G}}\right)_{e_i} \Psi +
\xi^{\Lie{G}}_{\xi^{\Lie{G}}_{e_i}e_i} \Psi  - \xi^{\Lie{G}}_{e_i}
(\xi^{\Lie{G}}_{e_i} \Psi).
$$
Moreover, if the $G$-structure is harmonic, then $ \nabla^* \nabla
\Psi = -\xi^{\Lie{G}}_{e_i} (\xi^{\Lie{G}}_{e_i} \Psi). $
\end{lemma}
\begin{proof}
Since $\nabla^G$ is a $G$-connection and $\Psi$ is stabilised
under the action of $G$, then $\nabla^G \Psi = 0$. Taking
$\nabla^G = \nabla + \xi^G$ into account, this implies that
$\nabla \Psi = - \xi^G \Psi$. Therefore,
$$
(\nabla^2\Psi)_{X,Y} = - \nabla_X (\xi^G_Y \Psi) +
\xi^G_{\nabla_XY}\Psi =  - \nabla^G_X (\xi^G_Y \Psi) + \xi^G_X
(\xi^G_Y \Psi)+ \xi^G_{\nabla_XY}\Psi.
$$
Because the presence of the metric $\langle \cdot , \cdot
\rangle$, any $(r,s)$-tensor field on $M$ is metrically equivalent
to a $(0,r+s)$-tensor field. Therefore,  we have only to make the
proof for covariant tensors fields. Thus, we can assume that
$\Psi$ is a $(0,s)$-tensor field on $M$. By a straightforward
computation we get {\footnotesize
\begin{gather*}
\nabla^G_X (\xi^G_Y \Psi)(Z_1, \dots, Z_s) =  - \sum_{i=1}^s X
\left(\Psi(Z_1,\dots, \xi^G_Y  Z_i, \dots, Z_s)\right) +
\sum_{i=1}^s \Psi(Z_1,\dots, \xi^G_Y \nabla^G_X  Z_i, \dots,
Z_s)\\[-0mm]
 + \sum^s_{i,j=1\atop i\neq j}\Psi(Z_1,\dots, \xi^G_Y
Z_i, \dots , \nabla^G_X Z_j, \dots, Z_s).
\end{gather*}}
Now using $\nabla^G \Psi = 0$, we have
\begin{eqnarray*}
 \sum_{i=1}^s X \left(\Psi(Z_1,\dots,
\xi^G_Y  Z_i, \dots, Z_s)\right) & = & \sum_{i=1}^s
\Psi(Z_1,\dots,
\nabla^G_X \xi^G_Y  Z_i, \dots, Z_s) \\
&& +\sum^s_{i,j=1\atop i\neq j}\Psi(Z_1,\dots, \xi^G_Y Z_i, \dots
, \nabla^G_X Z_j, \dots, Z_s).
\end{eqnarray*}
 Taking this identity into account in the expression for
$\nabla^G_X (\xi^G_Y \Psi)$, we will obtain $\nabla^G_X (\xi^G_Y
\Psi) = (\nabla^G_X \xi^G)_Y \Psi + \xi^{G}_{\nabla^G_X Y} \Psi$.
Therefore, for the second covariant derivative we get
$$
(\nabla^2\Psi)_{X,Y} =  - (\nabla^G_X \xi^G)_Y \Psi -
\xi^{G}_{\nabla^G_X Y} \Psi + \xi^G_{\nabla_XY}\Psi +  \xi^G_X
(\xi^G_Y \Psi),
$$
which proves the required expression for $\nabla^* \nabla \Psi.$
\end{proof}

\section{Harmonic almost Hermitian structures}{\indent}
\label{sect:almherm}
 An almost Hermitian manifold is a
$2n$-dimensional Riemannian manifold $(M,\langle \cdot, \cdot
\rangle)$ equipped with an  almost complex structure $J$
compatible with the metric, that is, $ J^{2} = -{\rm Id}$ and
$\langle JX,JY \rangle = \langle X,Y \rangle$, for all vector
fields $X$, $Y$. Associated to the almost Hermitian structure, the
two-form $\omega= \langle \cdot, J \cdot\rangle$, called the {\em
K\"ahler form}, is usually considered. Using $\omega$, $M$ can be
oriented by fixing a constant multiple of $\omega^n = \omega
\wedge \dots^{(n}\wedge \omega$ as volume form. Likewise, the
presence of an almost Hermitian structure is equivalent to say
that $M$ is equipped with a $\Lie{U}(n)$-structure. It is well
known that $\Lie{U}(n)$ is a closed and connected subgroup of
$\SO(2n)$ and $\SO(2n)/ \Lie{U}(n)$ is reductive; in fact, it is a
Riemannian symmetric space . Moreover, we have the decomposition
into $\Lie{U}(n)$-modules $\lie{so}(M) = \lie{u}(n)(M) \oplus
\lie{m}(M)$. We will omit the subindex $\sigma$ used in  previous
Sections. Also, as in references, we shall simply denote
$\lie{u}(n)(M)$ and $\lie{m}(M)$ by $\lie{u}(n)$ and
$\lie{u}(n)^{\perp}.$ The bundle $\lie{u}(n)$ (resp.,
$\lie{u}(n)^{\perp}$) consists of those skew-symmetric
endomorphisms $A$ on tangent vectors  such that $AJ=JA$ (resp.,
$AJ=-JA$). The identification $b_A( \cdot , \cdot ) = \langle A
\cdot , \cdot \rangle$ implies $\Lambda^2 \mbox{T}^* M \cong
\lie{so}(M)$. Therefore, $\Lambda^2 \mbox{T}^* M = \lie{u}(n)
\oplus \lie{u}(n)^{\perp}$, where in this case $\lie{u}(n)$
(resp., $\lie{u}(n)^{\perp}$) consists of those two-forms on $M$
which are Hermitian (resp., anti-Hermitian), i.e.,  $b(J\cdot , J
\cdot) = b(\cdot ,  \cdot)$ (resp., $b(J\cdot , J \cdot) = -
b(\cdot ,  \cdot)$).

The minimal $\Un(n)$-connection is given by $\nabla^{\Lie{U}(n)} =
\nabla + \xi^{\Lie{U}(n)}$, with
\begin{equation} \label{torsion:xi}
  \xi^{\Lie{U}(n)}_X Y = - \tfrac12 J\left( \nabla_X J \right) Y,
\end{equation}
(see~\cite{Falcitelli-FS:aH}). Moreover,  $\xi^{\Lie{U}(n)} \in
\mbox{T}^* M \otimes \un(n)^\perp$ is equivalent to the condition
$$
 \xi^{\Lie{U}(n)} J + J \xi^{\Lie{U}(n)} = 0.
$$
Since $\Un(n)$~stabilises the K\"ahler form~$\omega$, it follows
that $\nabla^{\Lie{U}(n)} \omega = 0$. Taking this into account,
$\xi^{\Lie{U}(n)} \in \mbox{T}^* M \otimes \un(n)^\perp$ implies
$\nabla \omega = - \xi^{\Lie{U}(n)} \omega \in \mbox{T}^* M
\otimes \un(n)^\perp$. Thus,
    one can
identify the $\Un(n)$-components of $\xi^{\Lie{U}(n)}$ with the
$\Un(n)$-components of $\nabla\omega$:
\begin{enumerate}
\item if $n=1$, $ \xi^{\Lie{U}(1)} \in \mbox{T}^* M \otimes
\un(1)^\perp = \{ 0 \}$; \item if $n=2$, $ \xi^{\Lie{U}(2)} \in
\mbox{T}^* M \otimes \un(2)^\perp = \Wc2 \oplus
  \Wc4$;
\item if $n \geqslant 3$, $ \xi^{\Lie{U}(n)} \in \mbox{T}^* M
\otimes \un(n)^\perp =
  \Wc1 \oplus \Wc2 \oplus \Wc3 \oplus \Wc4$.
\end{enumerate}
Here the summands~$\Wc{i}$ are the irreducible $\Un(n)$-modules
given by Gray and Hervella in~\cite{Gray-H:16}. In the following,
we will merely write $\xi = \xi^{\Lie{U}(n)}$ and $\xi_{(i)}$~will
denote the component in~$\Wc{i}$ of the intrinsic torsion~$\xi$.
For one-forms $\theta$, we will stand $J \theta (X) = -
\theta(JX)$, for all $X \in \mathfrak X(M)$. The one-form $J
\mbox{\it d}^* \omega$ is a constant multiple of the Lee one-form
which determines the $\mathcal W_4$-part of the intrinsic torsion
$\xi$ \cite{Gray-H:16}. Moreover, from \eqref{torsion:xi}, we will
have $ 2 \langle \xi_X Y , Z \rangle  = - (\nabla_X
\omega)(Y,JZ)$. Now, using this last identity, it is obtained that
the vector field $\xi_{e_i} e_i$ which take part in the
harmonicity criteria (see Theorem \ref{carharm})  is given by $ 2
\xi_{e_i} e_i = - J (d^* \omega)^{\sharp}$.

\begin{theorem} \label{characharmherm1}
For an almost Hermitian $2n$-manifold $(M,\langle \cdot, \cdot
\rangle ,J)$ with K\"ahler form $\omega$, we have that the
following conditions are equivalent:
 \begin{enumerate}
 \item[{\rm (i)}] The almost Hermitian structure is harmonic.
  \item[{\rm (ii)}] $[J,\nabla^* \nabla J] = 0$, where $[\cdot,\cdot]$ denotes the commutator bracket for endomorphisms.
 \item[{\rm (iii)}] $\nabla^* \nabla \omega$ is a Hermitian two-form.
 \item[{\rm (iv)}] $ \nabla^* \nabla \omega (X,Y)  =  - 4  \omega (\xi_{e_i} X,
\xi_{e_i}Y)$, for all $X,Y \in \mathfrak{X}(M)$.
 \end{enumerate}
\end{theorem}
\begin{remark}
{\rm Condition (ii) represents the Euler-Lagrange equations given
in \cite{Wood1} for the harmonic almost Hermitian structure
determined by  $J$}.
\end{remark}
\begin{proof}
Using Theorem \ref{carharm}, Lemma \ref{lapstaten} and $\xi J = -J
\xi$, it follows that (i) implies (iv) and (iv) implies
$$
\left( ( \nabla^{\Lie{U}(n)}_{e_i} \xi)_{e_i} + \xi_{\xi_{e_i}e_i}
\right) \omega =0.
$$
But note that the map $A \to -\omega(A \cdot, \cdot) -  \omega(
\cdot, A
 \cdot)$ from $\lie{u}(n)^{\perp} \subseteq \lie{so}(2n)$ to
 $\lie{u}(n)^{\perp} \subseteq \Lambda^2 \mbox{T}^* M$  is an
 $U(n)$-isomorphism. Therefore, $( \nabla^{\Lie{U}(n)}_{e_i} \xi)_{e_i} +
\xi_{\xi_{e_i}e_i}=0$.

Taking into account that $(\nabla^{\Lie{U}(n)}_{e_i} \xi)_{e_i}
\omega, \; \xi^{\Lie{U}(n)}_{\xi_{e_i}e_i}  \omega $ belong to $
\lie{u}(n)^{\perp}$, the equivalence between (iii) and (iv) is an
immediate consequence of Lemma \ref{lapstaten} and $\xi J = -J
\xi$.

Because we have $(\nabla_{X}\omega)(Y,Z) = \langle Y,
(\nabla_{X}J)Y\rangle$, it follows that
$$
(\nabla^{*}\nabla \omega)(X,Y) = \langle  X , (\nabla^* \nabla
J)Y\rangle.
$$
This implies the equivalence between (ii) and (iii).
\end{proof}

Tricerri and Vanhecke \cite{Tricerri-Vanhecke:aH} gave a complete
decomposition of the Riemannian curvature tensor~$R$ of an almost
Hermitian manifold $M$ into irreducible $\Un(n)$-components. These
divide naturally into two groups, one forming the
space~$\Kah=\Cur(\un(n))$ of algebraic curvature tensors for a
K\"ahler manifold (characterised by $\xi=0$), and the other,
$\Kah^\perp$, being its orthogonal complement. Additionally,
Falcitelli et al.~\cite{Falcitelli-FS:aH} showed that the
components of~$R$ in~$\Kah^\perp$ are linearly determined by the
covariant derivative $\nabla \xi$.   By using the minimal
$\Un(n)$-connection $\nabla^{\Lie{U}(n)}$ of~$M$, Falcitelli et
al.\ display some tables which show whether or not the tensors
$\nabla^{\Lie{U}(n)} \xi_{(i)}$ and $\xi_{(i)}\odot\xi_{(j)}$
contribute to the components of~$R$ in~$\Kah^\perp$.  Some
variations of such tables have been given in \cite{FMCAS}.
Explanations for these variations are based in Equation
\eqref{d2omega:part20} given below. All of this has provided a
unified approach to many of the curvature results obtained by
Gray~\cite{Gray:curvature}.

For studying some components of $R$, it is necessary to consider
the usual Ricci curvature tensor  $\Ric$, associated to the metric
structure, and another tensor $\Ric^*$, called the {\it
$\ast$-Ricci curvature tensor}, associated to the almost Hermitian
structure and defined by $\Ric^* (X,Y) = \langle R_{X , e_i} JY ,
Je_i\rangle$.

 In general, $\Ric^*$ is not symmetric. However, because
$
 \Ric^*(JX,JY) = \Ric^*(Y,X)$,
  it
can be claimed that  its Hermitian part  coincides with its
symmetric part $\Ric^*_{\mbox{\footnotesize s}}$, and its
anti-Hermitian part is equal to its skew-symmetric part
$\Ric^*_{\mbox{\footnotesize alt}}$. Under the action of
$\Lie{U}(n)$, $\Ric^*$ is decomposed into $\Ric^* =
\Ric^*_{\mbox{\footnotesize s}} + \Ric^*_{\mbox{\footnotesize
alt}} $, where $\Ric^*_{\mbox{\footnotesize s}} \in \mathbb R
\langle \cdot , \cdot \rangle \oplus \lie{su}(n)_s \subseteq S^2
\mbox{T}^* M$ and $ \Ric^*_{\mbox{\footnotesize alt}} \in
\lie{u}(n)^{\perp} \subseteq \Lambda^2 \mbox{T}^* M$
\cite{Tricerri-Vanhecke:aH}. Because in the present work the
tensor $\Ric^*_{\mbox{\footnotesize alt}}$ will play a special
r{\^o}le, we recall the following result.
\begin{lemma}[\cite{FMCAS}] \label{astricciah}
 If $M$ be is an almost Hermitian $2n$-manifold with
  minimal $\Un(n)$-connection $\nabla^{\Lie{U}(n)}  = \nabla + \xi$, then
  the skew-symmetric part $\Ric^*_{\mbox{\footnotesize alt}}$ of the
$*$-Ricci tensor is given by
  \begin{equation}
    \label{otraricsh}
    \begin{split}
     \Ric^*_{\mbox{\rm \footnotesize alt}} (X,Y)
      & = - \inp{\xi_{J\xi_{e_i} e_i}JX}Y
      +  \inp{(\nabla^{\Lie{U}(n)}_{e_i}\xi)_{Je_i}JX}Y.
    \end{split}
\end{equation}
\end{lemma}

From the fact $d^2 \omega=0$, writing $d^2 \omega$ by means of
$\nabla^{\Lie{U}(n)}$ and $\xi$, the identity
\begin{equation}
  \label{d2omega:part20}
  {\rm 
  \begin{array}{rl}
    0 =
    &3\inp{(\nabla^{\Lie{U}(n)}_{e_i}\xi_{(1)})_{e_i}X}Y
    - \inp{(\nabla^{\Lie{U}(n)}_{e_i}\xi_{(3)})_{e_i}X}Y
    + (n-2) \inp{(\nabla^{\Lie{U}(n)}_{e_i}\xi_{(4)})_{e_i}X}Y
    \\[2mm]
    &
    + \inp{{\xi_{(3)}}_Xe_i}{{\xi_{{(1)}e_i}}Y}
    - \inp{{\xi_{(3)}}_Ye_i}{{\xi_{(1)}}_{e_i}X}
    + \inp{{\xi_{(3)}}_Xe_i}{{\xi_{{(2)}e_i}}Y}
    - \inp{{\xi_{(3)}}_Ye_i}{{\xi_{{(2)}e_i}}X}
    \\[1mm]
    &
    - \displaystyle \frac{n-5}{n-1}\inp{{\xi_{{(1)} \xi_{{(4)} e_i} e_i}}X}Y
    - \displaystyle \frac{n-2}{n-1}\inp{{\xi_{{(2)}\xi_{{(4)} e_i} e_i}}X}Y
    +  \inp{{\xi_{{(3)}\xi_{{(4)} e_i} e_i}}X}Y
  \end{array} }
\end{equation}
was deduced in \cite{FMCAS}. Here, we will make use of
\eqref{d2omega:part20} below. Likewise, we need to point out that
${\xi_{{(4)} \, \xi_{e_i} e_i}} = 0$. In fact, this directly
follows from the expression for $\xi_{(4)}$ \cite{Gray-H:16} given
by
\begin{equation} \label{torsionw4}
 \langle   \xi_{{(4)}X} Y, JZ \rangle = -
\frac{1}{4(n-1)}
   \left\{     X^{\flat}  \wedge  d^* \omega (Y,Z)  -JX^{\flat} \wedge  Jd^* \omega (Y,Z)
   \right\}.
\end{equation}

Some results  proved in \cite{Wood1} are recovered in Theorem
\ref{classhermharm} below which is completed with other additional
results. In proving those results next Lemma will be useful.
\begin{lemma} \label{previo} For an almost Hermitian  $2n$-manifold $(M,\langle \cdot
,\cdot \rangle, J)$,  we have
\begin{eqnarray*}
2(n-1)   \langle (\nabla^{U(n)}_{e_i} \xi_{(4)})_{e_i} X , Y
\rangle & = & d (\xi^{\flat}_{e_i} e_i ) (X,Y) - d
(\xi^{\flat}_{e_i} e_i) (JX,JY) \\
 & & - 4 \langle \xi_{(1)\xi_{e_i} e_i } X , Y \rangle + 2 \langle  \xi_{(2)\xi_{e_i} e_i } X , Y
 \rangle .
\end{eqnarray*}
\end{lemma}
\begin{proof} From the expression  \eqref{torsionw4} we have
\begin{eqnarray} \label{w4expre}
2(n-1)    \xi_{{(4)}X}  & = &      X^\flat \otimes   \xi_{e_i} e_i
-  \xi^{\flat}_{e_i} e_i   \otimes  X - JX^\flat \otimes  J
\xi_{e_i} e_i +  J \xi^{\flat}_{e_i} e_i \otimes  JX.
\end{eqnarray}
Now, fixing   a local orthonormal frame field $\{ e_1 , \ldots ,
e_{2n} \}$ such that $(\nabla_{e_i} e_j )_m =0$, for a given  $m
\in M$, we will compute $(\nabla_{e_i} \xi_{(4)})_{e_i} X)_m$. In
fact, by  a straightforward computation we will obtain
\begin{eqnarray*}
2(n-1)   \langle (\nabla_{e_i} \xi_{(4)})_{e_i} X , Y \rangle  & =
&  d (\xi^{\flat}_{e_i} e_i ) (X,Y) - d (\xi^{\flat}_{e_i} e_i)
(JX,JY)    \\
  &&  + 2 \langle  \xi_{JX} JY - \xi_{JY} J X  , \xi_{e_i} e_i ,
  \rangle.
\end{eqnarray*}
Then, taking the properties of $\xi_{(i)}$ given in
\cite{Gray-H:16} into account, we will get
\begin{eqnarray*}
  \langle \xi_{{(4)}X} Y - \xi_{{(4)}Y} X , \xi_{e_i} e_i
\rangle   & = & 0.
\end{eqnarray*}
Thus, we will obtain the identity
\begin{eqnarray*}
2(n-1)   \langle (\nabla_{e_i} \xi_{(4)})_{e_i} X , Y \rangle  & =
& d (\xi^{\flat}_{e_i} e_i ) (X,Y) - d (\xi^{\flat}_{e_i} e_i)
(JX,JY) - 4 \langle \xi_{(1)\xi_{e_i} e_i } X , Y \rangle
     \\
      &&   + 2 \langle  \xi_{(2)\xi_{e_i} e_i } X , Y  \rangle
        + 2 \langle  \xi_{(3)X } Y - \xi_{(3)Y } X, \xi_{e_i} e_i
        \rangle.
\end{eqnarray*}
Finally, it is not hard to show
$$
 2(n-1)   \langle (\xi_{e_i} \xi_{(4)})_{e_i} X ,
Y \rangle = - 2 \langle  \xi_{(3)X } Y - \xi_{(3)Y } X  ,
\xi_{e_i} e_i  \rangle .
$$
From the last two identities, the required identity in Lemma
follows.
\end{proof}

\begin{theorem} \label{classhermharm}
 For an almost Hermitian  $2n$-manifold $(M,\langle \cdot
,\cdot \rangle, J)$,  we have:
\begin{enumerate}
  \item[{\rm (i)}] If $M$ is of type $\mathcal{W}_1 \oplus
\mathcal{W}_2 \oplus \mathcal{W}_4$, then the almost Hermitian
structure is harmonic if and only if
\begin{eqnarray*}
 \qquad \quad (n-1) \Ric_{alt}^* (X,Y) & = &   d (\xi^{\flat}_{e_i} e_i ) (X,Y) - d (\xi^{\flat}_{e_i}
e_i) (JX,JY) + 2 (n-3) \langle  \xi_{(1)\xi_{e_i} e_i } X , Y
\rangle \\
 & &  + 2n  \langle  \xi_{(2)\xi_{e_i} e_i } X , Y \rangle .
\end{eqnarray*}
 \item[{\rm (ii)}] If $M$ is quasi-K{\"a}hler  $(\mathcal{W}_1 \oplus
\mathcal{W}_2)$, then the almost Hermitian structure  is harmonic
if and only if $\Ric^*_{\mbox{\rm \footnotesize alt}} = 0$.

\item[{\rm (iii)}]
 If $M$ is locally conformal  almost K{\"a}hler $(\mathcal{W}_2
\oplus \mathcal{W}_4)$, then the almost Hermitian structure is
harmonic
 if and only if
\begin{equation*}
(n-1) \Ric^*_{\mbox{\rm \footnotesize alt}}(X,Y) =  2n \langle
\xi_{\xi_{e_i} e_i} X, Y \rangle,
\end{equation*}
for all $X,Y \in \mathfrak X (M)$.

 \item[{\rm (iv)}] If $M$ is
of type $\mathcal{W}_1 \oplus \mathcal{W}_4$ and $n \neq 2$, then
the almost Hermitian structure is harmonic  if and only if
\begin{equation*} 
(n-1)(n-5) \Ric^*_{\mbox{\rm \footnotesize alt}}(X,Y) =  2(n+1)(
n-3)\langle \xi_{ \xi_{e_i} e_i} X, Y \rangle,
\end{equation*}
for all $X,Y \in \mathfrak X (M)$.

  \item[{\rm (v)}] If $M$  is
Hermitian $(\mathcal{W}_3 \oplus \mathcal{W}_4)$, then the almost
Hermitian structure  is harmonic
 if and only if
\begin{equation*}
 \Ric^*_{\mbox{\rm \footnotesize alt}} (X,Y) =  - 2 \langle
\xi_{\xi_{e_i} e_i} X , Y \rangle.
\end{equation*}
\end{enumerate}
 In particular:
 \begin{enumerate}
  \item[\rm (i)$^*$] A nearly K{\"a}hler structure
$(\mathcal{W}_1)$ is a harmonic map.
  \item[\rm (ii)$^*$] If the exterior derivative of the  Lee form is Hermitian $($in particular, if it is closed$)$, a Hermitian
  structure is harmonic if and only if $\Ric^*_{\mbox{\rm \footnotesize
  alt}}=0$.
 \item[\rm (iii)$^*$] A balanced Hermitian structure $(\mathcal{W}_3)$ is a harmonic almost Hermitian structure.
  \item[\rm (iv)$^*$] A locally conformal K{\"a}hler  structure $(\mathcal{W}_4)$ is
a harmonic  almost Hermitian structure. In such a case, the Lee
form is closed and, therefore,   $ \Ric^*_{\mbox{\rm \footnotesize
alt}} =0$.
   \end{enumerate}
\end{theorem}

\begin{proof}
For (i). By  Lemma \ref{astricciah}, using  the properties of
$\xi_{(i)}$ given in \cite{Gray-H:16}, we have
 \begin{equation}
    \begin{split}
     \Ric^*_{\mbox{\rm \footnotesize alt}} (X,Y)
      & = \inp{\xi_{(1)\xi_{e_i} e_i}X}Y
           + \inp{\xi_{(2)\xi_{e_i} e_i}X}Y
      -  \inp{(\nabla^{\Lie{U}(n)}_{e_i}\xi_{(1)})_{e_i}X}Y\\
   & \quad
      -  \inp{(\nabla^{\Lie{U}(n)}_{e_i}\xi_{(2)})_{e_i}X}Y
      +  \inp{(\nabla^{\Lie{U}(n)}_{e_i}\xi_{(4)})_{e_i}X}Y.
    \end{split} \nonumber
\end{equation}
Now,  by Theorem \ref{carharm} and Lemma \ref{previo}, (i)
follows. In particular, if the structure is nearly K\"{a}hler,  by
Equation \eqref{d2omega:part20}, we have
 $(\nabla^{\Lie{U}(n)}_{e_i}\xi)_{e_i} =0$. Thus,
we get $\Ric^*_{\mbox{\rm \footnotesize alt}}=0$. Finally, by
Proposition \ref{pro:skew} (ii),  (i)$^*$ follows.

 Parts (ii) and (iii) are immediate consequences of (i).  We recall that,
 in case of locally conformal almost K\"{a}hler manifolds,  the Lee one-form is
closed. This fact is well known. In particular, if the structure
is locally conformal K\"{a}hler, then
$\inp{(\nabla^{\Lie{U}(n)}_{e_i}\xi)_{e_i}X}Y =0$ by Lemma
\ref{previo}. Moreover, we will also have $\inp{\xi_{(4)\xi_{e_i}
e_i}X}Y =0$. Then (iv)$^*$ follows.

 For (iv). Because the structure is of type $\mathcal{W}_1 \oplus
\mathcal{W}_4$, Equation \eqref{d2omega:part20} and Equation
\eqref{otraricsh} are respectively given by
\begin{equation}
  \label{d2omega:part20w1w4}
  \begin{array}{rl}
    0
    =& \;
3\inp{(\nabla^{\Lie{U}(n)}_{e_i}\xi_{(1)})_{e_i}X}Y +
     (n-2)
    \inp{(\nabla^{\Lie{U}(n)}_{e_i}\xi_{(4)})_{e_i}X}Y
      - \frac{n-5}{n-1}\inp{{\xi_{{(1)} \xi_{{(4)}e_i} e_i}}X}Y
   ,
  \end{array}
\end{equation}
\begin{equation}
    \label{otraricshw1w4}
    \begin{array}{rl}
      \Ric^*_{\mbox{\rm\footnotesize alt}} (X,Y)
      = &  \inp{\xi_{{(1)}\xi_{{(4)}e_i} e_i}X}Y
      -\inp{(\nabla^{\Lie{U}(n)}_{e_i}\xi_{(1)})_{e_i}X}Y
      +\inp{(\nabla^{\Lie{U}(n)}_{e_i}\xi_{(4)})_{e_i}X}Y      .
    \end{array}
\end{equation}
Likewise, the characterising  condition for harmonic almost
Hermitian  structures given in Theorem \ref{carharm} is expressed
by
\begin{equation} \label{characharmhermw1w4}
- \langle \xi_{{(1)} \xi_{{(4)}e_i} e_i} X , Y \rangle = \langle (
\nabla^{\Lie{U}(n)}_{e_i}
 \xi_{(1)})_{e_i} X , Y \rangle  + \langle (
\nabla^{\Lie{U}(n)}_{e_i}
 \xi_{(4)})_{e_i} X , Y \rangle .
\end{equation}
Now, for $n \geq 3$, it is straightforward to check that Equation
\eqref{d2omega:part20w1w4},  Equation \eqref{otraricshw1w4} and
Equation \eqref{characharmhermw1w4} imply the expression for
$\Ric^*_{\mbox{\rm \footnotesize alt}}$ required in (iv).

Reciprocally, it is also direct to see that such an expression for
$\Ric^*_{\mbox{\rm \footnotesize alt}}$, Equation
\eqref{d2omega:part20w1w4} and Equation \eqref{otraricshw1w4}
 imply Equation \eqref{characharmhermw1w4}. Therefore, the almost
 Hermitian structure is harmonic.

For (v). The intrinsic torsion $\xi$ for Hermitian structures is
such that $\xi_{JX} JY = \xi_{X} Y$ \cite{Gray-H:16}. Therefore,
the required identity in (v) is an immediate consequence of
Theorem \ref{carharm} and Lemma \ref{astricciah}.

For (ii)$^*$. By Lemma \ref{previo}, if the exterior derivative of
the  Lee form is Hermitian, then
             $(\nabla^{\Lie{U}(n)}_{e_i}\xi_{(4)})_{e_i} = 0$ in
             this case. But we also have
             $(\nabla^{\Lie{U}(n)}_{e_i}\xi_{(3)})_{e_i}= \xi_{{(3)} \xi_{e_i} e_i}$  by
             \eqref{d2omega:part20}. Therefore, the assertion is a
             consequence of (v).

For (iii)$^*$. Now, we have $\xi_{e_i} e_i=0$. Moreover, Equation
\eqref{d2omega:part20} implies $(
\nabla^{\Lie{U}(n)}_{e_i}\xi_{(3)})_{e_i} = 0$.
\end{proof}

\begin{example}{\rm It is well-known that a $3$-symmetric space
$(M,\langle\cdot,\cdot\rangle)$ admits a canonical almost complex
structure $J$ compatible with $\langle\cdot,\cdot\rangle$ and
$(M,\langle\cdot,\cdot\rangle, J)$ becomes into a quasi-K{\"a}hler
manifold. Further, the intrinsic torsion $\xi=
-\frac{1}{2}J(\nabla J)$ of the corresponding $U(n)$-structure is
a homogeneous structure (see for example \cite{Sato}). Hence,
$\xi$ is $\nabla^{U(n)}$-parallel and then we get
$\Ric^*_{\mbox{\rm \footnotesize alt}}= 0.$ Then, from Theorem
\ref{classhermharm} (ii), we can concluse that} the canonical
almost Hermitian structure of a $3$-symmetric space is harmonic.
\end{example}

If we write $\langle \xi_{(1)X} Y , Z \rangle = \Psi_{\xi}
(X,Y,Z)$, then  $\Psi_{\xi}$ is a skew-symmetric three-form such
that $\Psi_{\xi} (JX,JY,Z)= -\Psi_{\xi} (X,Y,Z)$ \cite{Gray-H:16}.
For $n \geq 3$, if we have a harmonic almost Hermitian structure
of type $\mathcal W_1 \oplus \mathcal W_4$, then it follows, using
Theorem \ref{characharmherm1} (iv) and Equation \eqref{torsionw4},
that the connection Laplacian of $\omega$ is given by
$$
\nabla^* \nabla \omega (X,Y) = 4  \langle X \lrcorner \Psi_{\xi} ,
JY \lrcorner \Psi_{\xi} \rangle + \frac{1}{4(n-1)^2} d^* \omega
\wedge J d^* \omega (X,Y).
$$
Note that, in general, the right side of this equality is not
collinear with $\omega$. In particular, if $n=3$, we obtain
$$
\nabla^* \nabla \omega = \frac{ \| \Psi_{\xi} \|^2}{36} \omega +
\frac{1}{16} d^* \omega \wedge  J d^* \omega.
$$

A harmonic  section $\sigma$ into a sphere  bundle  of a
Riemannian vector bundle is  characterised by the condition
$\nabla^*\nabla \sigma = \frac{\| \nabla \sigma\|^2}{\| \sigma
\|^2} \sigma$ or, equivalently, $\nabla^*\nabla \sigma$ is
collinear with $\sigma$ (see \cite{GMS}, \cite{Salvai}). From the
previous paragraphs, the first part of next result is immediate.
\begin{proposition} For six-dimensions,
the nearly K\"{a}hler structures are the only harmonic almost
Hermitian structures of type $\Wc1 + \Wc4$, such that $\omega$ is
also  a  harmonic  section into a sphere bundle in $\Lambda^2 T^*
M$. For four-dimensions, locally conformal K\"{a}hler structures
implies that  $\omega$ is   a  harmonic section into  a sphere
bundle in $\Lambda^2 T^* M$.
\end{proposition}
\begin{proof} Let $M$ be a locally conformal K\"{a}hler four-manifold.
 In order to compute $(\nabla^* \nabla \omega)_m$, for $m \in M$, we
will consider a local orthonormal frame field $\{ e_1 , \ldots ,
e_4 \}$ such that $(\nabla_{e_i} e_j )_m =0$. Thus, because in
this case,  $\nabla_X \omega = X^\flat \wedge (\theta^\sharp
\lrcorner \omega) - \theta \wedge (X \lrcorner \omega)$, where
$\theta  = \frac12 J d^* \omega = - \xi_{e_i} e_i$
\cite{Gray-H:16}, we have
$$
(\nabla^* \nabla \omega)_m = - e_i^\flat \wedge  (\theta^\sharp
\lrcorner (\nabla_{e_i} \omega)) - e_i^\flat \wedge (
(\nabla_{e_i} \theta)^\sharp \lrcorner \omega) + \nabla_{e_i}
\theta \wedge (e_i \lrcorner \omega) - \theta \wedge d^* \omega.
$$
Now, using the expression for $\nabla \omega$ and the identities
$e_i \wedge (e_i \lrcorner \omega) = 2 \omega$ and $e_i^\flat
\wedge \theta \wedge (\theta^\sharp \lrcorner ( e_i \lrcorner
\omega))=  \theta \wedge (\theta^\sharp \lrcorner \omega)$, we
obtain
$$
 - e_i^\flat \wedge  (\theta^\sharp
\lrcorner (\nabla_{e_i} \omega)) = - 2 \theta \wedge (
\theta^\sharp \lrcorner \omega) + 2 \| \theta \|^2 \omega.
$$
Moreover, because $\theta$ is closed, we have $(\nabla_X
\theta)(Y) = (\nabla_Y \theta)(X)$ and  it is not hard to see
$$
 e_i^\flat \wedge ( (\nabla_{e_i}
\theta)^\sharp \lrcorner \omega) = \nabla_{e_i} \theta \wedge (e_i
\lrcorner \omega).
$$
Finally, from all of this and $d^* \omega =- 2 \theta^\sharp
\lrcorner\omega$, we get $\nabla^* \nabla \omega = 2 \| \theta
\|^2 \omega $.
\end{proof}

\begin{remark}{\rm
For nearly K\"{a}hler connected six-manifolds which are not K\"{a}hler,
 if $5 \alpha$ denotes the Einstein constant and using
\cite[Equation (3.10)]{FMCAS}, we have $$ \nabla^* \nabla \omega
(X,Y) =
  4  \langle \xi_{e_i} X,
\xi_{e_i}JY \rangle = 4 \alpha \, \omega(X,Y).
$$
Therefore, $\| \Psi_{\xi} \|^2= 144 \alpha$.

On the other hand, for locally conformal K\"{a}hler four-manifolds, we
have  $\nabla^* \nabla \omega = 2 \| \theta \|^2 \omega $.
Therefore, $ \frac1{16}  \| \nabla \omega \|^2 =  \frac12 \|
\theta \|^2 = \frac12 \| \xi_{e_i} e_i \|^2 = \frac18  \| J d^*
\omega\|^2 $ that, in general,  it is not constant. }
\end{remark}

 In \cite{BHLS}, Bor et al. have shown diverse results relative to the energy of almost
Hermitian structures defined on  certain compact Riemannian
manifolds. Concretely, they prove the following
\begin{theorem}[\cite{BHLS}] \label{bor-Hlam-Salva} Let
$(M^{2n},\langle \cdot , \cdot \rangle )$ be a compact Riemannian
manifold such that
\begin{enumerate}
 \item[$\bullet$] $n \geq 3$ and $(M,\langle \cdot , \cdot \rangle
 )$is conformally flat, or
 \item[$\bullet$] $n =2$ and $(M,\langle \cdot , \cdot \rangle
 )$ is anti-self-dual.
\end{enumerate}
Then an orthogonal almost complex structure $J$ on $M$ is an
energy minimiser in each one of the following three cases:
\begin{enumerate}
\item[\rm (i)] $n=3$ and $J$ is of type $\mathcal W_1 \oplus
\mathcal W_4$.
 \item[\rm (ii)] $n=2$ or $n\geq 4$ and $J$ is of
type $\mathcal W_4$.
 \item[\rm (iii)] $n$ arbitrary and $J$ is of type $\mathcal W_2$.
\end{enumerate}
\end{theorem}

 Because $\Ric_{\mbox{\footnotesize alt}}^*$ determines certain
$\Lie{U}(n)$-component, $n \geq 2$, of the Weyl curvature tensor
$W$ on almost Hermitian  (see \cite{Falcitelli-FS:aH,FMCAS}), then
we have $\Ric_{\mbox{\footnotesize alt}}^*=0$  for almost
Hermitian $2n$-manifolds which  are locally conformal flat. In
particular, for $n=2$, if we  consider the action $\SO(4)$
determined by the volume form given  by $Vol =  \frac12 \omega
\wedge \omega$, the Weyl curvature tensor is decomposed into two
components, that is, $W= W^+ + W^-$. If $W^+=0$ ($W^-=0$), the
manifold is called {\it anti-self-dual} ({\it self-dual}). More
details can be found in \cite{Salamon,Falcitelli-FS:aH}.
 Since $\Ric_{\mbox{\footnotesize alt}}^*$
determines certain $\Lie{U}(2)$-component of $W^+$, if the
manifold is anti-self-dual, then we will also have
$\Ric_{\mbox{\footnotesize alt}}^*=0$. Therefore, it follows that
the results here presented are in agreeing with Theorem
\ref{bor-Hlam-Salva}. \vspace{2mm}

Now, we focus  attention on harmonicity as a map of almost
Hermitian structures. Results in that direction were already
obtained in \cite{Wood2}, we will complete such results by using
tools here presented. In next Lemma, $s^*$ will denote the {\it
$*$scalar  curvarture} defined by $s^* = \Ric^* (e_i, e_i)$. If
$\Ric^* (X,Y) = \frac{1}{2n} s^* \langle X,Y \rangle$, then the
almost Hermitian manifold is said to be {\it weakly $*$Einstein}.
If $s^*$ is constant, a weakly-$*$Einstein manifold is called {\it
$*$Einstein}.

 In Riemannian geometry, it is satisfied $2 d^* \Ric + d s =0$,
 where $s$ is the scalar curvature. The $*$analogue in almost
 Hermitian geometry is false. In fact, this is  clarified by the
 following two results.
\begin{lemma} \label{id:genera}
For  almost Hermitian manifolds, we have {\rm
$$
 2 d^* \Ric^{*t}(X) + ds^* (X) =   2 \langle R{(e_i, X)} ,    \xi_{Je_i} J \rangle
        -  4 \Ric^* ( X,   \xi_{e_i} e_i ) + 4 \langle \Ric^*  , \xi^{\flat}_X  \rangle,
$$}
where $\Ric^{*t}(X,Y) = \Ric^{*}(Y,X)$ and $\xi_X^{\flat} (Y,Z) =
\langle \xi_X Y, Z \rangle$. In particular, if the manifold is
weakly $\ast$Einstein, then
$$
 \frac{n-1}{n} d s^*(X) =    2 \langle R{( e_i, X)} ,    \xi_{Je_i} J
 \rangle - 2 s^* \langle \xi_{e_i} e_i , X \rangle.
$$
\end{lemma}
\begin{proof}
Note that $ \Ric^{*t}(X,Y) = \frac12  \langle R{(e_i,Je_i)} Y,
JX\rangle$.  Then, we get
\begin{eqnarray*}
d^* \Ric^{*t}(X) & = & - (\nabla_{e_j} \Ric^{*t}) ( e_j ,X)
\\
&=& -\frac12  e_j \langle R{(e_i,Je_i)} X ,Je_j \rangle
   + \frac12 \langle R{(e_i,Je_i)}  \nabla_{e_j} X , Je_j \rangle
   + \frac12 \langle R{(e_i,Je_i)}   X,   J\nabla_{e_j} e_j \rangle
   \\
   &=& - \frac12  \langle ( \nabla_{e_j} R){(e_i,Je_i)} X ,Je_j \rangle
      -  \langle R{(\nabla_{e_j} e_i,Je_i)}   X , Je_j \rangle
        - \frac12 \langle R{(e_i,Je_i)}   X,   (\nabla_{e_j}J) e_j \rangle
\end{eqnarray*}
Now,  by symmetric properties of $R$ and $\xi = - \frac12 J(\nabla
J)$, it follows that
\begin{eqnarray*}
d^* \Ric^{*t}(X) & = & - \frac12  \langle ( \nabla_{e_j} R){(X
,Je_j)}  e_i,Je_i\rangle
      +  \langle R{(X , e_j)} e_i ,    \nabla_{Je_j} Je_i) \rangle
        -  \langle R{(e_i,Je_i)}   X,  J \xi_{e_j} e_j \rangle.
\end{eqnarray*}
Using second Bianchi's identity and taking
$$
 \langle R{(X , e_j)} e_i ,    \nabla_{Je_j} Je_i) \rangle =
   \langle R{(X , e_j)} e_i ,    \nabla^{\Lie{U}(n)}_{Je_j} Je_i) \rangle
 -  \langle R{(X , e_j)} e_i ,    \xi_{Je_j} Je_i) \rangle
$$
into account, we get
\begin{eqnarray*}
d^* \Ric^{*t}(X) & = & - \frac14  \langle ( \nabla_{X} R){(e_j,
Je_j )}  e_i,Je_i\rangle
      -  \langle R{(X , e_j)} ,    \xi_{Je_j} J \rangle
        -  2 \Ric^* ( X,   \xi_{e_j} e_j ) .
\end{eqnarray*}
Note that
$$
 \langle R{(X , e_j)} e_i ,    \nabla^{\Lie{U}(n)}_{Je_j} Je_i) \rangle =  \langle R{(X , e_j)} e_i , e_k \rangle
     \langle \nabla^{\Lie{U}(n)}_{Je_j} Je_i , e_k \rangle =0,
$$
because it is a scalar product of a skew-symmetric matrix by a
Hermitian symmetric matrix.

Finally, it is obtained
\begin{eqnarray} \label{uno}
\qquad 2d^* \Ric^{*t}(X) & = & - \frac12  \langle ( \nabla_{X}
R){(e_j, Je_j )}  e_i,Je_i\rangle
      - 2 \langle R{(X , e_j)} ,    \xi_{Je_j} J \rangle
        -  4 \Ric^* ( X,   \xi_{e_j} e_j ) .
\end{eqnarray}

In a  second instance, $ ds^* (X) =  \frac12  X \langle
R{(e_i,Je_i)} e_j, Je_j \rangle$. Hence, we get
\begin{eqnarray*}
 ds^* (X) & = & \frac12   \langle (\nabla_X R){(e_i,Je_i)} e_j, Je_j \rangle
  + 2 \langle R{(e_i,Je_i)} e_j, \nabla_X Je_j \rangle.
\end{eqnarray*}
But we have also that
\begin{eqnarray*}
 \langle R{(e_i,Je_i)} e_j, \nabla_X Je_j \rangle & = &  \langle R{(e_i,Je_i)} e_j, e_k \rangle \langle
 \nabla^{\Lie{U}(n)}_X Je_j , e_k \rangle - \langle R{(e_i,Je_i)} e_j, e_k \rangle \langle
 \xi_X Je_j , e_k \rangle \\
 & = &  \langle R{(e_i,Je_i)} e_j, J\xi_X e_j \rangle =  2
 \Ric^* (e_i , \xi_X e_i) =  2 \langle \Ric^*  , \xi_X  \rangle.
\end{eqnarray*}
Thus, it follows that
 \begin{equation} \label{dos}
  ds^* (X)  =  \frac12   \langle (\nabla_X
R){(e_i,Je_i)} e_j, Je_j \rangle + 4 \langle \Ric^*  , \xi_X
\rangle.
 \end{equation}
 From \eqref{uno} and \eqref{dos}, the required identity is
 obtained.
\end{proof}

\begin{theorem}
 For an almost Hermitian  $2n$-manifold $(M,\langle \cdot
,\cdot \rangle, J)$,  we have:
\begin{enumerate}
 \item[{\rm (i)}] If $M$ is of type $\mathcal{W}_1 \oplus
\mathcal{W}_2 \oplus \mathcal{W}_4$, then the almost Hermitian
structure is a harmonic map  if and only if the almost Hermitian
structure is harmonic and
\begin{eqnarray*}
\qquad  (n-1) d^* \Ric^{*t}(X) + \frac{n-1}{2} ds^* (X) & = & \Ric
( X, \xi_{e_i} e_i )
        -  (2n-1) \Ric^* ( X,   \xi_{e_i} e_i )\\
        &&  + 2 (n-1) \langle \Ric^*  , \xi^{\flat}_X  \rangle,
\end{eqnarray*}
 for all  $X \in \mathfrak{X}(M)$.
 \item[{\rm (ii)}] If $M$ is quasi-K{\"a}hler  $(\mathcal{W}_1 \oplus
\mathcal{W}_2)$, then the almost Hermitian structure  is a
harmonic map  if and only if $\Ric^*$ is symmetric and $2 d^*
\Ric^{*}+ ds^*=0$. In particular,
 if the quasi-K{\"a}hler
manifold is weakly-$\ast$Einstein, then the almost Hermitian
structure is a  harmonic map  if and only if $s^*$ is constant.

  \item[{\rm (iii)}] If $M$  is
Hermitian $(\mathcal{W}_3 \oplus \mathcal{W}_4)$, then the almost
Hermitian structure  is a harmonic map
 if and only if  $\Ric^*_{\mbox{\rm \footnotesize alt}}  =  - 2 \xi^{\flat}_{\xi_{e_i}
 e_i}$
and
$$
2 d^* \Ric^{*t}(X) + ds^* (X) +  4 \Ric^* ( X,   \xi_{e_j} e_j ) -
4 \langle \Ric^*  , \xi^{\flat}_X  \rangle =0,
$$
for all $X \in \mathfrak X (M)$. In particular:
  \begin{enumerate}
  \item[{\rm (a)$^*$}] If the exterior derivative of the  Lee form is Hermitian $($in particular, if it is closed$)$,
   then the Hermitian
  structure is a harmonic map if and only if $\Ric_{\mbox{\footnotesize \rm alt}}^* =0$ and
   $2 d^* \Ric^{*} + ds^* +  4 \xi_{e_i} e_i  \lrcorner
  \Ric^* = 0$.
  \item[{\rm (b)$^*$}] If $\Ric^*$ is symmetric,  then the Hermitian
  structure is a harmonic map if and only if   $\xi_{\xi_{e_i} e_i}=0$  and
   $2 d^* \Ric^{*} + ds^* +  4 \xi_{e_i} e_i  \lrcorner
  \Ric^* = 0$. In particular, if the manifold is weakly-$\ast$Einstein, then the Hermitian
  structure is a harmonic map if and only if $\xi_{\xi_{e_i} e_i}=0$ and
   $(n-1) ds^* +  2 s^*  \xi^{\flat}_{e_i} e_i  = 0$.
   \item[{\rm (c)$^*$}] If the manifold is balanced Hermitian
   $(\mathcal{W}_3)$, then the almost Hermitian structure is a
   harmonic map if and only if $2 d^* \Ric^{*} + ds^*=0$.
   Furthermore, if the  balanced Hermitian manifold is
   weakly-$*$Einstein, the almost Hermitian
structure is a  harmonic map  if and only if $s^*$ is constant.
    \item[{\rm (d)$^*$}] If the manifold is locally conformal K\"{a}hler
   $(\mathcal{W}_4)$, then the almost Hermitian structure is a
   harmonic map if and only if $2 d^* \Ric^{*} + ds^* +  4 \xi_{e_i} e_i  \lrcorner
  \Ric^*=0$ if and only if, for all $X \in \mathfrak{X}(M)$, $(\Ric -\Ric^{*}) (X , \xi_{e_i} e_i) =0$.
  \end{enumerate}
\end{enumerate}
\end{theorem}
\begin{proof} All results contained in Theorem are immediate
consequences of Theorem \ref{classhermharm},  Lemma
\ref{id:genera}, and the following  consequence of the expression
for $\xi_{(4)}$ given by \eqref{w4expre}
$$
(n-1) \langle \xi_{(4) e_i} , R_{(e_i,X)} \rangle  = (\Ric
-\Ric^{*}) (X , \xi_{e_i} e_i).
$$
\end{proof}

\begin{example}{\rm
Hopf manifolds are diffeomorphic to $S^1 \times S^{2n-1}$ and
admit a locally conformal K\"{a}hler structure with parallel Lee form
$\frac{1}{2(n-1)} \xi_{e_i}^\flat e_i$ \cite{Va}. Furthermore,
$\xi_{e_i} e_i$ is nowhere zero and tangent to $S^1$. The metric
on $S^1 \times S^{2n-1}$ is the product metric of constant
multiples of the metrics on $S^1$ and $S^{2n-1}$ induced by the
respective Euclidean metrics on $\mathbb R^2$ and $\mathbb
R^{2n}$. The set $\mathcal L ( S^{2n-1})$ will consist of those
vector fields on $S^1 \times S^{2n-1}$ which are lifts of vector
fields on $S^{2n-1}$. The Riemannian curvature tensor $R$ is such
that
 \begin{gather*} \langle R{(X,Y)} Z_1
, Z_2 \rangle = k ( \langle X,Z_1 \rangle \langle Y,Z_2 \rangle -
\langle X,Z_2 \rangle \langle Y,Z_1 \rangle), \qquad R{(X,
\xi_{e_i}e_i)} =0,
 \end{gather*}
for all $X,Y,Z_1,Z_2 \in \mathcal L ( S^{2n-1})$, where $k$ is a
constant. Therefore, $\langle R{(e_i ,  \xi_{e_i} e_i)} ,
\xi_{e_i} \rangle =0$. Moreover, using the expression given by
\eqref{w4expre}, for all $X \in \mathcal L ( S^{2n-1})$, we have
 $$
 \langle R{(e_i , X)} , \xi_{e_i}
\rangle = 2k  \langle \xi_{e_i} e_i , X\rangle  =0.
 $$
 Additionally, it can be checked that
 $$
 \frac{n-1}{k} \langle \xi_X , R{(Y,Z)} \rangle = - J X^\flat \wedge J  \xi_{e_i}^\flat e_i
 (Y,Z),
 $$
 for all $X,Y,Z$ orthogonal to $ \xi_{e_i} e_i$. Therefore, the almost
 Hermitian structure is not horizontally geodesic. As a consequence, it is also not a flat structure.

Finally, using again the expression  \eqref{w4expre} and the fact
that $\xi_{e_i} e_i$ is parallel, it is obtained
\begin{eqnarray*} 
2(n-1)^2    (\nabla_X  \xi)_X
   & = &
      J \xi^\flat_{e_i} e_i (X) \; (X^\flat \otimes J \xi_{e_i} e_i
      -  J \xi^\flat_{e_i} e_i  \otimes  X + JX^\flat \otimes    \xi_{e_i} e_i
 -       \xi^\flat_{e_i} e_i \otimes  JX).
\end{eqnarray*}
Note that this expression is not vanished for all $X$.

In conclusion}, the locally conformal K\"ahler structure on $S^1
\times
 S^{2n-1}$ is a harmonic map which is neither horizontally geodesic, nor vertically
 geodesic.
\end{example}

\begin{example}
{\rm In general, locally conformal K\"{a}hler structures are not
harmonic maps. In fact, one can consider the K\"ahler structure on
$\mathbb R^{2n}$ determined by the Euclidean metric $\langle \cdot
, \cdot \rangle$ and the standard almost complex structure $J$. If
we do a conformal change of metric  using a function $f$ on
$\mathbb R^{2n}$, the new metric $\langle \cdot , \cdot \rangle_o
= e^f \langle \cdot , \cdot \rangle$ and $J$ determine a new
almost Hermitian structure which is locally conformal K\"ahler.
The  Lee form for the new structure is $df$ and the Riemannian
curvature tensor is given by
\begin{eqnarray*}
 -2 e^{-f} \langle R_o (X,Y)Z,W \rangle_o & = &   L(X,Z) \langle Y, W \rangle +
L(Y,W) \langle X , Z \rangle \\
&&
 - L(X,W) \langle Y, Z \rangle - L(Y,Z) \langle X, W \rangle
\\
&& + \frac{\|d f \|^2}{2} \{ \langle X, Z \rangle \langle Y, W
\rangle - \langle Y, Z \rangle \langle X, W \rangle \},
\end{eqnarray*}
where $L(X,Y)= (\nabla_X df) (Y) - \frac12 df(X) df(Y)$ and
$\nabla$ is the Levi-Civita connection associated to $\langle
\cdot , \cdot \rangle$ (see \cite{Tricerri-Vanhecke:aH}). If
$\xi_o$ denotes the intrinsic torsion of the structure $(J,
\langle \cdot , \cdot \rangle_o)$, an straightforward computation
shows that
\begin{equation*}
16 e^f  \langle R_o (e_{o \,i} ,X) , \xi_{o e_{o \,i}} \rangle_o =
- \frac{2n-3}{2}  d (\| df \|^2)(X) + d^* (df) df(X)
 +  ( \nabla_{JX} df )(J \mbox{grad}\,f ),
\end{equation*}
where $\{ e_{o \,1} , \dots , e_{o \,2n} \}$ is an orthonormal
basis for vectors with respect to $\langle \cdot , \cdot
\rangle_o$ and the terms in the right side, the norm $\| \cdot
\|$, $\mbox{grad}$, etc.,  are considered  with respect to the
Euclidean metric $\langle \cdot , \cdot \rangle$. Therefore, it is
not hard to find  functions $f$ such that $\langle R_o (e_{o \,i}
,X) , \xi_{o e_{o \,i}} \rangle_o \neq 0$. For instance, if $f=
\sin x_1$, then $\langle R_o (e_{o \,i} ,X) , \xi_{o e_{o \,i}}
\rangle_o = \frac{n-1}{8}\,  e^{-\sin x_1} \sin x_1 \, \cos x_1 \,
d x_1$.

If we take the function $f$ such that $(x^{i}\comp f)(x) =
(x^{i}\comp f)(x + 2\pi),$ $i=1,\dots ,2n,$ then
$\langle\cdot,\cdot\rangle_{o}$ determines a Riemannian metric on
the torus $T^{2n} = S^{1}\times \dots \times S^{1}$ and the
natural projection of $\mathbb R^{2n}$ on $T^{2n}$ becomes into a
local isometry. Hence, we also get} locally conformal K\"{a}hler
structures which are not harmonic maps on the torus
$(T^{2n},\langle\cdot,\cdot\rangle_{o}).$
\end{example}

\subsection{Nearly K\"{a}hler manifolds} For completeness, here we will
give a detailed and self-contained explanation of the situation
for nearly K\"{a}hler manifolds. Thus, we will recover results
already known originally proved, some of them, by Gray and,
others, by Wood. However, we will display alternative proofs in
terms of the intrinsic torsion $\xi$. Additionally, it is also
shown that, for nearly K\"{a}hler manifolds, $\xi$ is parallel
with respect to the minimal connection $\nabla^{\Lie{U}(n)}$,
i.e.,  $\nabla^{\Lie{U}(n)} \xi =0$. This last result is
originally due to Kirichenko \cite{Kir}.

The intrinsic torsion $\xi$ of a nearly K\"{a}hler manifold is
characterised by the condition $\xi_X Y= - \xi_Y X$. Because this
property is preserved by the action of $\Lie{O}(2n)$, then we have
also  $(\nabla_X \xi)_Y Z = - (\nabla_X  \xi)_Z Y$ and $(\nabla_X
\xi)^{\Lie{U}(n)}_Y Z = - (\nabla^{\Lie{U}(n)}_X \xi)_Z Y$.
Moreover, with respect to the almost complex structure $J$, it is
also satisfied $\xi_{JX} JY= - \xi_X Y$. Therefore, $(\nabla_X
\xi)^{\Lie{U}(n)}_{JY} JZ = - (\nabla^{\Lie{U}(n)}_X \xi)_{Y} Z$.

For nearly K\"{a}hler manifolds, Gray \cite{Gray:spheres} showed that
the following identities are satisfied
\begin{eqnarray} \label{ecxy}
\langle R{(X,Y)} X, Y \rangle - \langle R{(X,Y)} JX, JY \rangle  & = & 4 \| \xi_X Y \|^2  \\
\langle R{(JX,JY)} JZ, JW \rangle & = & \langle R{(X,Y)} Z, W
\rangle. \label{ecjxjy}
\end{eqnarray}
In fact, since $\langle (\nabla_{X} \xi)_Y X , Y \rangle = 0 =
\langle (\nabla_{Y} \xi)_X X , Y \rangle$, it is immediate that
$$
\langle R{(X,Y)\lie{u}(n)^{\perp}} X , Y \rangle  = \frac12 \left(
\langle R{(X,Y)} X , Y \rangle - \langle R{(X,Y)} JX , JY \rangle
\right) = 2 \langle [\xi_X , \xi_Y] X , Y \rangle.
$$
From this, \eqref{ecxy} follows. Also \eqref{ecxy} follows from
Proposition  \ref{pro:skew}, because $[ \lie{u}(n)^{\perp},
\lie{u}(n)^{\perp}] \subseteq  \lie{u}(n)$.

 For \eqref{ecjxjy}.  Using
\eqref{ecxy}, it is not hard to prove $\langle R{(JX,JY)} JX, JY
\rangle  = \langle R{(X,Y)} X, Y \rangle$. Then, by linearizing,
we will have \eqref{ecjxjy}.
\begin{theorem} Nearly K\"{a}hler structures are vertically geodesic
harmonic maps. Moreover, for nearly K\"{a}hler manifolds,  we have
\begin{eqnarray} \label{ecxyzw}
\langle R{(X,Y)} Z, W \rangle - \langle R{(X,Y)} JZ, JW \rangle  & = & 4 \langle \xi_X Y , \xi_Z W \rangle, \\
 \nabla^{\Lie{U}(n)}_X \xi & = &0. \label{nparallel}
\end{eqnarray}
 In particular, if the nearly K\"{a}hler structure is
flat, then is K\"{a}hler.
\end{theorem}
\begin{remark}{\rm Equation \eqref{ecxyzw} is due to Gray
\cite{Gray:nearly}. On the other hand,  Wood proved  in
\cite{Wood2} that nearly K\"{a}hler structures are vertically geodesic
harmonic maps.}
\end{remark}
\begin{proof} Since $\xi_X Y = - \xi_Y X$ and $\nabla \xi =
\nabla^{\Lie{U}(n)}\xi - \xi \xi$, it is direct to show that
\begin{eqnarray} \label{rjjxixi}
\langle R{(X,Y)\lie{u}(n)^{\perp}} Z  , W \rangle
 & = &  \frac12
\left( \langle R{(X,Y)} Z , W \rangle - \langle R{(X,Y)} JZ , JW
\rangle \right) \\
& = & \langle (\nabla^{\Lie{U}(n)}_X \xi)_Y Z , W \rangle -
\langle (\nabla^{\Lie{U}(n)}_Y \xi)_X Z , W \rangle
   + 2 \langle  \xi_X Y , \xi_{Z} W \rangle. \nonumber
\end{eqnarray}
Now, we consider the map ${\sf s} : \Lambda^2 T^* M \otimes
\Lambda^2 T^* M \to S^2 (\Lambda^2 T^* M)$ defined by ${\sf
s}(a\otimes b) = a \otimes b + b \otimes a$ and the map ${\sf b} :
S^2 (\Lambda^2 T^* M) \to S^2 (\Lambda^2T^* M)$ defined by
$$
{\sf b}(\Upsilon)(X,Y,Z,W) = 2 \Upsilon(X,Y,Z,W) -
\Upsilon(Z,X,Y,W) - \Upsilon(Y,Z,X,W).
$$
Applying the composition ${\sf b} \circ {\sf s}$ to both sides of
Equation \eqref{rjjxixi} and, then, making use of \eqref{ecjxjy}
and first Bianchi's identity, we will obtain
\begin{eqnarray}  \label{sbianchi}
&& \\
 3 \langle R{(X,Y)} Z  , W \rangle - 2 \langle R{(X,Y)} JZ , JW
\rangle && \nonumber \\
  +  \langle R{(Z, X)} JY  , JW \rangle
    +  \langle R{(Y, Z)} JX  , JW \rangle
  & = &
    8 \langle \xi_{Z} W , \xi_X Y \rangle
   +4 \langle \xi_{Y} W , \xi_X Z \rangle
   -4 \langle \xi_{X} W , \xi_Y Z \rangle. \nonumber
\end{eqnarray}
Note that we have also taken $(\nabla_X^{\Lie{U}(n)} \xi)_Y Z = -
(\nabla_X^{\Lie{U}(n)} \xi)_Z Y $ into account. Now, if  we
replace $Z$ and $W$ by $JZ$ and $JW$, subtract the result from
\eqref{sbianchi} and  use $\xi_{JX} JY= - \xi_X Y$, then we get
\begin{eqnarray} \label{ultimopaso}
5 \langle R{(X,Y)} Z  , W \rangle - 5\langle R{(X,Y)} JZ , JW
\rangle   & &
\\
 -  \langle R{(X,JY)} Z  , JW \rangle  -  \langle R{(X, JY)} JZ  , W \rangle
  & = &
   16  \langle \xi_{X} Y , \xi_Z W \rangle. \nonumber
\end{eqnarray}
Finally,  replacing in \eqref{ultimopaso} $Y$ and $Z$ by $JY$ and
$JZ$, multiplying by $1/5$ the resulting equation  and adding the
final result to \eqref{ultimopaso}, the required identity
\eqref{ecxyzw} is obtained. \vspace{2mm}

Now, from the following identity
$$
\frac12 \left( \langle R{(X,Y)} X, Z \rangle - \langle R{(X,Y)}
JX, JZ \rangle \right) = \langle (\nabla_{X} \xi)_Y X , Z \rangle
- \langle (\nabla_{Y} \xi)_X X , Z \rangle +  2 \langle [\xi_X ,
\xi_Y] X , Z \rangle,
$$
using \eqref{ecxyzw}, we get $(\nabla_{X} \xi)_X = 0$. Hence the
nearly K\"{a}hler structure is vertically geodesic. \vspace{2mm}

For \eqref{nparallel}. Since
\begin{eqnarray*}
 \langle R{(X,Y)} Z , W \rangle - \langle R{(X,Y)} JZ , JW
\rangle & = & 2 \langle (\nabla^{\Lie{U}(n)}_X \xi)_Y Z , W
\rangle - 2 \langle (\nabla^{\Lie{U}(n)}_Y \xi)_X Z , W \rangle
\\
&& + 4 \langle \xi_X Y , \xi_{Z} W \rangle
\\ & = & 4 \langle \xi_X
Y , \xi_{Z} W \rangle,
\end{eqnarray*}
we have $(\nabla^{\Lie{U}(n)}_X \xi)_Y = (\nabla^{\Lie{U}(n)}_Y
\xi)_X$. Moreover, from the identity
\begin{eqnarray*}
\langle  (\nabla_X \xi)_Y  Z , W \rangle
 & = & \langle (\nabla^{\Lie{U}(n)}_X \xi)_Y Z , W \rangle   + \langle \xi_{Z} W , \xi_X Y \rangle
   - \langle
  \xi_X \xi_Y Z , W \rangle   + \langle \xi_Y \xi_X Z , W \rangle,
\end{eqnarray*}
taking $(\nabla_{X} \xi)_X = 0$ into account, it follows
$(\nabla^{\Lie{U}(n)}_{X} \xi)_X = 0$.  Therefore,
$(\nabla^{\Lie{U}(n)}_X \xi)_Y = (\nabla^{\Lie{U}(n)}_Y \xi)_X = -
(\nabla^{\Lie{U}(n)}_X \xi)_Y =0$.

 The final remark contained in Theorem follows from Proposition
 \ref{pro:skew} (i).
\end{proof}


\begin{thebibliography}{99}

\setlength{\baselineskip}{0.4cm}
\bibitem{Besse:Einstein}
A.~L.~Besse, \emph{{E}instein manifolds}, Ergebnisse der
Mathematik und ihrer
  Grenzgebiete, 3. Folge, vol.~10, Springer, Berlin, Heidelberg and New York,
  1987.





\bibitem{BHLS} G.~Bor, L.~Hern{\'a}ndez Lamoneda and M.~Salvai, {\em Orthogonal
almos-complex structures of minimal energy on conformally, and
half-conformally flat manifolds}, Geom. Dedicta (to appear). {\tt
arXiv:math.DG/0609511}





\bibitem{EeLe1} J.~Eells and L.~Lemaire, A report on harmonic maps,
\emph{Bull. London Math. Soc.}  10 (1978), 1-68.

\bibitem{EeLe2} \bysame and \bysame, Another report on harmonic maps, \emph{Bull. London
Math. Soc.} 20 (1988), 385-524.

\bibitem{EeSa} \bysame
 and J.~H.~Sampson, Harmonic mappings of
Riemannian manifolds, \emph{Amer. J. Math.} 86 (1964), 109-160.

\bibitem{CleytonSwann:torsion}
R.~Cleyton and A.~F.~Swann, {E}instein metrics via intrinsic or
parallel
  torsion, \emph{Math. Z.} 247 no.~3(2004), 513--528.
\bibitem{Falcitelli-FS:aH}
M.~Falcitelli, A.~Farinola and S.~M. Salamon, Almost-{H}ermitian
  geometry, \emph{Differential Geom. Appl.} 4 (1994), 259--282.

\bibitem{GGV} 
  O.~Gil-Medrano, J.~C.~Gonz\'{a}lez-D\'{a}vila and L.~Vanhecke,
 Harmonicity and minimality of oriented distributions,
\emph{Israel Journal of Math.} {\rm 143} (2004), 253-279.
\bibitem{GMS} J.~C.~Gonz\'{a}lez-D\'{a}vila, F.~Mart\'{\i}n-Cabrera and
M.~Salvai, Harmonicity of sections on sphere bundles, preprint,
(2007).

\bibitem{Gray:spheres}
A.~Gray, Almost complex submanifolds of the six sphere,
\emph{Proc. Amer. Math. Soc.}  \textrm{20} (1970),  277--279.

\bibitem{Gray:nearly}
\bysame, Nearly K\"{a}hler Manifolds, \emph{J. Diff. Geom.} 4 (1976),
  283--309.

\bibitem{Gray:curvature}
\bysame, Curvature identities for {H}ermitian and almost
{H}ermitian
  manifolds, \emph{T\^ohoku Math. J.} (2) \textrm{28} (1976), no.~4,
  601--612.



\bibitem{Gray-H:16}
\bysame and L.~M.~Hervella, \textrm{The sixteen classes of almost
{H}ermitian
  manifolds and their linear invariants}, \emph{Ann. Mat. Pura Appl.} (4) \textrm{123}
  (1980), 35--58.

\bibitem{Kir} V.~F.~Kirichenko, \emph{K-spaces of maximal
rank}, Mat. Zametki 22 (1977), 465--476.

  \bibitem {LM} B.~Lawson and M.~L.~Michelsohn, \emph{Spin Geometry}, Princeton
  University Press, 1989.












\bibitem{FMCAS}
F.~Mart{\'\i}n Cabrera and A.~Swann: \textrm{Curvature of Special
Almost Hermitian
  Manifolds}, \emph{Pacific J. Math.} 228 (2006), 165--184. {\tt
arXiv:math.DG/0501062}

 \bibitem{Sak} T. Sakai, \emph{Riemannian Geometry}, Transl. Math. Mon. 149, Amer. Math. Soc., Providence (1996).

\bibitem{Salamon} {\rm S. Salamon}, \emph{Riemannian Geometry and
Holonomy Groups}, Pitman Research Notes in Math. Series, {\bf
201}, Longman (1989).

\bibitem{Salvai} M.~Salvai, On the energy of sections of trivializable
sphere bundles, \emph{Rendiconti del Seminario Mathematico della
Universit\'{a} e Politecnico di Torino} {\rm 60} (2002), 147-155.

\bibitem{Sato} T.~Sato, Riemannian $3$-symmetric spaces and
homogeneous K-spaces, \emph{Mem. Fac. of Technology Kanazawa
Univ.} {\rm 12(2)} (1979), 137-143.

\bibitem{Tricerri-Vanhecke:aH}
F.~Tricerri and L.~Vanhecke, \textrm{Curvature tensors on almost
{H}ermitian
  manifolds}, \emph{Trans. Amer. Math. Soc.} \textrm{267} (1981),
  365--398.

\bibitem{Ur} H.~Urakawa, \emph{Calculus of variations and harmonic maps},
Transl. of Math. Mon.  132, Amer. Math. Soc., Providence, Rhode
Island, 1993.


\bibitem{Va}  I.~Vaisman, Locally conformal K\"{a}hler manifolds with parallel Lee
form, \emph{Rend. Mat.} (6) 12 (1979), no. 2, 263--284.


\bibitem{Vilms} J.~Vilms, Totally geodesic maps, \emph{J. Diff. Geom.} 4
(1970), 73--79.
\bibitem{Wie1} G.~Wiegmink, Total bending of vector fields on
Riemannian manifolds, \emph{Math. Ann.} 303 (1995), 325-344.



\bibitem{Wood1} C.~M.~Wood, \textrm{Harmonic almost-complex
structures}, \emph{Composition Mathematica} 99 (1995), 183-212.
\bibitem{Wood} \bysame, \textrm{On the energy of a unit vector field}, \emph{Geom. Dedicata}  64 (1997), 319--330.

\bibitem{Wood2} \bysame, \textrm{Harmonic sections of homogeneous fibre bundles},
\emph{Differential Geom. Appl.} 19 (2003), 193-210.
\end{thebibliography}
\end{document}